\documentclass[12pt]{amsart}

\usepackage[all]{xy}
\usepackage{amssymb}
\usepackage{amsmath}
\usepackage{stmaryrd}
\usepackage{mathabx}

\usepackage[OT2,T1]{fontenc}
\DeclareSymbolFont{cyrletters}{OT2}{wncyr}{m}{n}
\DeclareMathSymbol{\Sha}{\mathalpha}{cyrletters}{"58}

\theoremstyle{definition}
\newtheorem*{prop*}{Proposition}

\theoremstyle{definition}
\newtheorem*{thm*}{Theorem}

\theoremstyle{definition}
\newtheorem*{cor*}{Corollary}

\theoremstyle{definition} 
\newtheorem{ssprop}[subsubsection]{Proposition}

\theoremstyle{definition} 
\newtheorem{sprop}[subsection]{Proposition}
\newcommand{\sProposition}[2]{\begin{sprop} \label{#1} 
{#2} \end{sprop}}

\theoremstyle{definition} 
\newtheorem{sconj}[subsection]{Conjecture}

\theoremstyle{definition} 
\newtheorem{sthm}[subsection]{Theorem}

\theoremstyle{definition} 
\newtheorem{ssthm}[subsubsection]{Theorem}

\theoremstyle{definition} 
\newtheorem{sslm}[subsubsection]{Lemma}

\theoremstyle{definition} 
\newtheorem{slm}[subsection]{Lemma}

\theoremstyle{definition} 
\newtheorem{scor}[subsection]{Corollary}

\theoremstyle{definition} 
\newtheorem{sremark}[subsection]{Remark}

\theoremstyle{definition} 
\newtheorem{ssconj}[subsubsection]{Conjecture}

\theoremstyle{definition}
\newtheorem*{conj*}{Conjecture}

\theoremstyle{definition} 
\newtheorem{sscond}[subsubsection]{Condition}

\theoremstyle{definition} 
\newtheorem{suntitled}[subsection]{}
\newcommand{\sUntitled}[2]{\begin{suntitled} \label{#1} 
{#2} \end{suntitled}}

\let\oldmarginpar\marginpar
\renewcommand\marginpar[1]{\-\oldmarginpar[\raggedleft\footnotesize #1]%
{\raggedright\footnotesize #1}}

\input cyracc.def
\DeclareFontFamily{U}{russian}{}
\DeclareFontShape{U}{russian}{m}{n}
        { <5><6> wncyr5
        <7><8><9> wncyr7
        <10><10.95><12><14.4><17.28><20.74><24.88> wncyr10 }{}
\DeclareSymbolFont{Russian}{U}{russian}{m}{n}
\DeclareSymbolFontAlphabet{\mathcyr}{Russian}
\makeatletter
\let\@math@cyr\mathcyr
\renewcommand{\mathcyr}[1]{\@math@cyr{\cyracc #1}}
\makeatother
\newcommand{\sh}{\mathcyr{sh}} 

\newcommand{\set}[2]{\big\{ #1 \; \big| \; #2 \big\} }

\newcommand{\Vect}{\operatorname{\bf{Vect}}}

\newcommand{\from}{\leftarrow}
\newcommand{\xto}{\xrightarrow}
\newcommand{\xfrom}{\xleftarrow}
\newcommand{\surj}{\twoheadrightarrow}

\newcommand{\gr}{\operatorname{gr}}

\newcommand{\Spec}{\operatorname{Spec}}

\newcommand{\Li}{\operatorname{Li}}

\newcommand{\Hom}{\operatorname{Hom}}
\newcommand{\Ext}{\operatorname{Ext}}

\newcommand{\Aut}{\mathrm{Aut}\,}

\newcommand{\ad}{\operatorname{ad}}
\newcommand{\Lie}{\operatorname{Lie}}

\newcommand{\m}[1]{\mathrm{#1}}
\newcommand{\fk}[1]{\mathfrak{#1}}

\newcommand{\bb}[1]{\mathbb{#1}}

\newcommand{\la}{\lambda}
\newcommand{\ka}{\kappa}
\newcommand{\si}{\sigma}
\newcommand{\Si}{\Sigma}
\newcommand{\ze}{\zeta}
\newcommand{\ga}{\gamma}
\newcommand{\al}{\alpha}
\newcommand{\be}{\beta}
\newcommand{\om}{\omega}
\newcommand{\Om}{\Omega}
\newcommand{\ep}{\epsilon}
\newcommand{\de}{\delta}
\newcommand{\up}{\upsilon}

\newcommand{\Ga}{{\mathbb{G}_a}}
\newcommand{\Gm}{{\mathbb{G}_m}}
\newcommand{\Qp}{{\QQ_p}}
\newcommand{\Zp}{{\ZZ_p}}

\newcommand{\ZZ}{\bb{Z}}

\newcommand{\GG}{\mathbb{G}}
\newcommand{\nN}{\fk{n}}

\newcommand{\NN}{\bb{N}}
\newcommand{\QQ}{\bb{Q}}

\newcommand{\PP}{\bb{P}}

\newcommand{\VV}{\bb{V}}

\newcommand{\Uu}{\mathcal{U}}
\newcommand{\uU}{\mathfrak{u}}

\newcommand{\Bb}{\mathcal{B}}
\newcommand{\FF}{\bb{F}}
\renewcommand{\AA}{\bb{A}}

\newcommand{\Oo}{\mathcal{O}}

\newcommand{\Ss}{\mathcal{S}}
\newcommand{\Aa}{\mathcal{A}}
\newcommand{\Ff}{\mathcal{F}}
\newcommand{\Rr}{\mathcal{R}}

\newcommand{\Gg}{\mathcal{G}}
\newcommand{\Ee}{\mathcal{E}}
\newcommand{\Pp}{\mathcal{P}}
\newcommand{\Tt}{\mathcal{T}}

\newcommand{\aA}{\fk{a}}

\newcommand{\inv}{^{-1}}

\newcommand{\areq}{\ar@{=}}
\newcommand{\suphook}{\ar@{^(->}}
\newcommand{\subhook}{\ar@{_(->}}

\newcommand{\smses}[6]
{
\[
\xymatrix{
1 \ar[r] &
#1 \ar[r]_-{#2} &
#3 \ar[r]_-{#4} &
#5 \ar[r] \ar@/_1.5pc/[l]_-{#6} &
1
}
\]
}

\newcommand{\inj}{\hookrightarrow}

\newcommand{\thrpl}{\PP^1 \setminus \{0,1,\infty\}}

\newcommand{\dR}{{\rm {dR}}}

\newcommand{\et}{{\textrm {\'et}}}

\newcommand{\sha}{\begin{tiny}{\Sha}\end{tiny}}

\newcommand{\PL}{\m{PL}}
\newcommand{\ev}{\fk{ev}}

\newcommand{\MT}{\operatorname{\mathbf{MT}}}

\newcommand{\logu}{\log^\uU}
\newcommand{\Liu}{\Li^\uU}
\newcommand{\zeu}{\ze^\uU}

\newcommand{\un}{\m{un}}

\newcommand{\wt}{\operatorname{wt}}

\newcommand{\Xo}{e_1}
\newcommand{\Xoo}{e_{11}}
\newcommand{\Xt}{e_2}
\newcommand{\Xtt}{e_{22}}
\newcommand{\Xot}{e_{12}}
\newcommand{\eo}{e_1}
\newcommand{\eoo}{e_{11}}
\renewcommand{\et}{e_2}
\newcommand{\ett}{e_{22}}
\newcommand{\eot}{e_{12}}

\newcommand{\Sym}{\operatorname{Sym}}

\newcommand{\opnm}{\operatorname}
\newcommand{\VIC}{\opnm{VIC}}
\newcommand{\VICinfty}{\VIC^{[\infty]}}

\newcommand{\new}{\newcommand}
\new{\Lisse}{\opnm{Lisse}}

\new{\ti}{\times}

\title
[$M_{0,5}$: Towards higher dimensions]{$M_{0,5}$: Towards the Chabauty-Kim method in higher dimensions}

\author{Ishai Dan-Cohen and David Jarossay}

\thanks{Both authors were supported by
 ISF grant 726/17.}

\date{\today}

\begin{document}

\maketitle

\begin{abstract}
If $Z$ is an open subscheme of $\Spec \ZZ$, $X$ is a sufficiently nice $Z$-model of a smooth curve over $\QQ$, and $p$ is a closed point of $Z$, the Chabauty-Kim method leads to the construction of locally analytic functions on $X(\Zp)$ which vanish on $X(Z)$; we call such functions ``Kim functions''. At least in broad outline, the method generalizes readily to higher dimensions. In fact, in some sense, the surface $M_{0,5}$ should be easier than the previously studied curve $M_{0,4} = \thrpl$ since its points are closely related to those of $M_{0,4}$, yet they face a further condition to integrality. This is mirrored by a certain \emph{weight advantage} we encounter, because of which, $M_{0,5}$ possesses \emph{new Kim functions} not coming from $M_{0,4}$. Here we focus on the case ``$\ZZ[1/6]$ in half-weight 4'', where we provide a first nontrivial example of a Kim function on a surface.

Central to our approach to Chabauty-Kim theory (as developed in works by S. Wewers, D. Corwin, and the first author) is the possibility of separating the geometric part of the computation from its arithmetic context. However, we find that in this case the geometric step grows beyond the bounds of standard algorithms running on current computers. Therefore, some ingenuity is needed to solve this seemingly straightforward problem, and our new Kim function is huge.

\end{abstract}

\setcounter{tocdepth}{1}
\tableofcontents

\raggedbottom
\SelectTips{cm}{11}

\section{Introduction}

The Chabauty-Kim method, introduced in \cite{kimi, kimii}, extends the classical Chabauty method in two (related) directions. By going to higher quotients of the fundamental group (where the Chabauty method stops at the abelianization) it produces $p$-adic analytic functions (``Kim functions'') which vanish on integral points, beyond the Chabauty bound. Thus, it can be applied in cases where the Chabauty method does not apply. However, even in cases where the Chabauty method does apply to produce a $p$-adic analytic function which can be used to bound the set of integral points, it rarely produces a sharp bound. As one climbs up the tower of unipotent quotients, however, the Chabauty-Kim method produces more functions. Together, these may be used to give a sharp bound. Indeed, according to Kim's conjecture \cite{nats}, this should be the case for (suitable integral models) of all hyperbolic curves.

Going exactly one step beyond the abelian quotient leads to the so-called \emph{quadratic Chabauty method}. In a growing number of cases \cite{BalakDogEffectiveCKThm,nats,BalakDogQuadI,BalakBesMullComputing,BalakBesMulQuadChab, BalakBesColGr, BalakAppendix, KimMassey} this has been worked out to produce numerical results, and those results have been used to provide numerical evidence for the conjecture. Of particular note is the work \cite{Balak13} whose final point-count (apart from verifying another case of the conjecture) solved an old and sought-after problem in arithmetic. 

The methods of Dan-Cohen--Wewers \cite{CKtwo, mtmue, mtmueii} and Corwin--Dan-Cohen \cite{PolGonI, PolGonII}, while so far limited to the simplest of all cases ($X = \thrpl$), have been particularly successful in going beyond the quadratic level. These articles incorporate the methods of mixed Tate motives and motivic iterated integrals (see, for instance, \cite{DelGon, GonGal, Brown, BrownDecomp}). A key point is the possibility of extracting the geometric aspects of the computation from their natural arithmetic surroundings. The result is an algorithm which includes among its subalgorithms a \emph{geometric step} and an \emph{arithmetic step}. The arithmetic step involves a search for enough motivic iterated integrals to generate suitable portions of the mixed Tate Hopf algebra, and its halting is conditional on conjectures of Goncharov, refined somewhat with respect to ramification. Before embarking on the present work, we regarded the geometric step as being comparatively simple, as it may, in principle, be solved by standard algorithms. 

Kim's method generalizes naturally to higher dimensions. The connection with the section conjecture suggests that a suitable generalization of Kim's conjecture may hold for anabelian varieties. In this article we take a conservative step in this direction. Kim's conjecture for $X = M_{0,4}$ over $Z \subset \Spec \ZZ$ implies Kim's conjecture for $Y = M_{0,5}$ (the moduli space of genus $0$ curves with $5$ marked points) over $Z$. Nevertheless, as we go up from $M_{0,4}$ to $M_{0,5}$, we encounter a \emph{weight advantage}, which allows us to construct Kim functions on $M_{0,5}(\Zp)$ not coming from $M_{0,4}$ (see \S\ref{addendum}).
Our first opportunity to take advantage of this weight advantage occurs for $\ZZ[1/6]$ in half-weight $4$, and it is this one case that we focus on in the present work. 

Our conclusions (so far) are somewhat mixed. Much of the work for $M_{0,4}$ generalizes readily. However, the geometric step via standard computational methods has turned out to be computationally infeasible. With a careful (but elementary) analysis of the geometric step (and a certain method involving resultants) we are nevertheless able to produce a new Kim function, which turns out to be huge. 

To state our result, let us recall Kim's method in outline. Our purpose here is only to fix notation and terminology, and we refer the reader for instance to \cite{nats} for a general introduction, and to \cite{mtmue} for our mixed Tate version. The Chabauty-Kim method applied to a smooth mixed-Tate variety $Y$ over $Z \subset \Spec \ZZ$, a prime of good reduction $p \in Z$, and a finite type $\Gm$-equivariant quotient
\[
\pi_1^\un(Y,b) \surj \pi'
\]
of the unipotent fundamental group of $Y$ at the $Z$-integral base-point $b$, revolves around a commuting diagram
\[
\xymatrix{
Y(Z) \ar[d]_\ka \ar[r] & Y(\Zp) \ar[d]^\al \\
H^1(\pi_1^{\MT}(Z, \dR), \pi'_\dR)_\Qp \ar[r]_-{LR}
 & \pi'_{\dR, \Qp}.
}
\]
which we refer to as \emph{Kim's cutter}. Here $\pi_1^{\MT}(Z, \dR)$ is the mixed Tate Galois group of $Z$ at the de Rham fiber functor, the decoration ``dR'' denotes de Rham realization, $\ka$ is the unipotent Kummer map, $\al$ is the unipotent Albanese map --- a morphism of $p$-adic analytic spaces, and $\m{LR}$, which is a map of finite-type $\Qp$-varieties (in our case, affine spaces) is obtained by a combination of localization and realization. The coordinate ring $A'$ of $\pi'$ is graded by the $\Gm$-action hiding inside the action of
\[
\pi_1^{\MT}(Z, \dR) = \Gm \ltimes \pi_1^\un(Z, \dR)
\]
on $\pi'_\dR$ and we refer to the graded degree of a function as its \emph{half-weight}. If $f \in A'_{\Qp, n}$ is a function of half-weight $n$ such that $\m{LR}^\sharp(f) = 0$ then $f^\m{BC}:=\al^\sharp f$ is a Besser-Coleman function on $Y(\Zp)$ which vanishes on $Y(Z)$. We refer to such a function as a \emph{$p$-adic Kim function for $Y$ over $Z$ in half-weight $n$}. Let $p$ be any prime not equal to $2$ or $3$.

\begin{sthm}
\label{fbctheorem}
The function $F^{BC}$ on $M_{0,5}(\Zp)$ constructed in \S\ref{waca2} below is a $p$-adic Kim function for $M_{0,5}$ over $\ZZ[1/6]$ in half-weight four\footnote{up to a small $p$-adic error}.
\end{sthm}

Unlike the unipotent fundamental group of $\thrpl$, which has been studied to death, working with the unipotent fundamental group of $M_{0,5}$ requires that we address some basic issues ourselves. This mostly concerns a certain analog of the polylogarithmic quotient. For instance, in \S\ref{polquosection}, we revisit (and generalize in a straightforward way) the beautiful proof due to Deligne \cite{Deligne89} and Deligne--Goncharov \cite{DelGon} that the polylogarithmic quotient is semisimple. We also give an explicit algebra-basis for the algebra of functions on our polylogarithmic quotient in \S\ref{ab0}.

Since the polylogarithmic quotient is not fixed under automorphisms of $M_{0,5}$, translation by automorphisms gives rise to inequivalent Kim functions. We may then ask if the vanishing locus of $F^\m{BC}$ and its conjugates is finite, or even equal to the set of integral points. Precedent for computations of this sort may be found in \cite{BalBesBiaMul,DograUnlikely,HastFuncTransc}. Unfortunately, the large size of $F^\m{BC}$ presents a hurdle to computation. We hope to face this challenge in a separate future work. 

The tower of moduli spaces of curves with marked points, and especially its first two steps $M_{0,4}$ and $M_{0,5}$, plays a central role in Grothendieck's vision for anabelian geometry, and, relatedly, in motivating relations between multiple zeta values (complex and $p$-adic) \cite{FurushoPenagonHexagon,FurushoJafari,BesserFurusho,OiUeno} (see also \cite[Chapter 25]{AndreIntro} and \cite{FresanGil} and the references there). We hope that further investigation may shed some light, in one direction, on the interaction between Kim's cutter and the tower. In another direction, we hope to better understand how the geometry of the tower controls relations between motivic iterated integrals. In turn, this may lead to a better understanding of the ramification of motivic iterated integrals and hence to more precise $S$-integral refinements of Goncharov's conjectures. As explained in \cite{mtmue, mtmueii, PolGonI, PolGonII}, our algorithms for $\thrpl$ rely on such statements for halting, and a better understanding will lead to faster and more elegant algorithms. As we explain in \cite{M0n}, our methods with resultants also help to clarify and simplify the geometric step for $\thrpl$ and for punctured lines in general. 

This article does not include introductory material on the moduli space $M_{0,5}$; the facts we use, which we learned for instance from \cite{GonchManin, BrownMZVandPeriodsOfM0n} as well as the references given above, are summarized in \S\ref{polquosection}. This article is also written in correct logical order, which, at least in this case, runs counter to the natural flow of exposition. Indeed, sections \ref{ap0}--\ref{sobsection} make no mention of $M_{0,5}$. Most readers will want to start with \S\ref{polquosection} and to refer back as needed. 

Speaking of \emph{order}, a word is in order concerning the order of multiplication in fundamental groups. For many reasons, it seems to us far more natural to let $\al \be$ denote first $\be$, then $\al$. For instance, this is the notation used in \emph{category theory}, which is why it is sometimes referred to as the ``functorial order''. The reverse ``lexicographic'' order, it seems to us, leads to the systematic reversal of a vast swath of mathematics. However, for reasons we do not understand, there appears to be quite a tradition of using lexicographic order, and the authors of \cite{PolGonI}, in particular, chose to follow this tradition. Thus, we're forced to resolve this conflict within the body of the article by using both orderings and spelling out where and how we transition between them. For this purpose, we prefer not to think of the question of ordering as being merely a matter of \emph{notation}. Rather, given a Tannakian category, we have a \emph{functorial fundamental groupoid} and a \emph{lexicographic fundamental groupoid}, and the one is the opposite of the other. To prevent any confusion, we systematically denote objects stemming from the lexicographic fundamental groupoid with a superscript `$\times$', which stands for the `x' in \textit{lexicographic}.

Finally, the reader may have noticed a footnote, according to which, for Theorem \ref{fbctheorem} to be precise we would need to bound the error incurred by our $p$-adic approximations. This task, while somewhat tedious, presents no particular difficulty. Since our purpose here is to demonstrate a method (its promise, and its challenges), we have chosen not to carry this out. If the reader is disturbed by this logical wrinkle, she may view the main result of this work as a fully fleshed-out algorithm which associates to every $\ep >0$ a Besser-Coleman function $F^\m{BC}_\ep$ on $M_{0,5}(\Zp)$ which is within $\ep$ of a Kim function. The particular function we construct is then an \emph{example} with $\ep$ fixed. As explained for instance in \cite{mtmueii}, such an algorithm suffices for the application to integral points and the verification of Kim's conjecture. 

In fact, we do not expect our formula for $F^\m{BC}_\ep$ in terms of polylogarithms to change as $\ep$ shrinks further. Proving this would require proving that the formulas for decomposition of certain motivic polylogarithms in terms of shuffle coordinates on the mixed Tate Galois group obtained via computations of certain $p$-adic periods in \S\ref{sobsection} hold precisely (and not only to within $\ep$). Some methods for doing so are demonstrated for instance in \cite{mtmue, PolGonI} (along with attribution to those who taught us these methods). But these are not needed for the application to integral points.

\subsection*{Acknowlegements}
We would like to thank Jennifer Balakrishnan, Amnon Besser and Hidekazu Furusho for their interest and encouragement as well as for helpful conversations. In particular, we thank Balakrishnan for her participation in our attempts to approach the ``geometric step'' via computer computation. Finally, we are grateful to the referee for many helpful comments and suggestions.

\section{The polylogarithmic Lie algebra in abstraction}
\label{ap0}

\subsection{} \label{1.1}
Let $L$ be the free Lie algebra on the set $\Gamma$ of generators
\[
e_1, e_{11}, e_2, e_{22}, e_{12}
\]
modulo the relations 
\[
\tag{R}
[e_1, e_2] = [e_{11}, e_2] = [e_1, e_{22}] = 0 
\]
\[
[e_{11}, e_{22}] = -[e_{11}, e_{12}] = [e_{22}, e_{12}] = 
[e_2-e_1, e_{12}].
\]
Let $N$ be the Lie ideal generated by $\Xoo, \Xtt, \Xot$. We define the \emph{polylogarithmic Lie algebra (for $M_{0,5}$)} by 
\[
L^{\PL} = L/[N,N].
\]
Note that in $L^{\PL}$ we have
\[
[\Xo, \Xot] = [\Xt, \Xot].
\]

\begin{slm} \label{echo1}
In $L^{\PL}$, we have
\[
\tag{*}
(\ad \Xo)^n(\Xot) = (\ad \Xt)^n(\Xot).
\]
\end{slm}

\begin{proof}
Since $\Xo, \Xt$ commute, we have, by the Jacobi identity, for all $Y$, 
\[
[\Xo,[\Xt,Y]] = [\Xt,[\Xo,Y]].
\]
Hence by induction 
\[\tag{**}
(\ad \Xo)^n([\Xt,Y]) = [e_2, (\ad \Xo)^n(Y)].
\]
We use (**) to establish (*) by induction on $n$.  As noted above, the case $n=1$ follows directly from the defining relations. Suppose the lemma holds up to $n \ge 1$. Then
\begin{align*}
(\ad \Xo)^{n+1}(\Xot) 
	&=(\ad \Xo)^{n}([\Xo, \Xot]) \\
	&= (\ad \Xo)^{n}([\Xt, \Xot]) \\
	&= [e_2, (\ad \Xo)^n(\Xot)] \\
	&= [e_2, (\ad \Xt)^n(\Xot)] \\
	&=  (\ad \Xt)^{n+1}(\Xot)
					\end{align*}
which establishes the induction step and hence the lemma.
\end{proof}

\begin{slm}\label{liebasis}
The Lie algebra $L^{\PL}$ has vector space basis $S=$
\[
\Xo, \; \Xt, \; (\ad \Xo)^n(\Xoo), \; (\ad \Xt)^n(\Xtt), \;
(\ad \Xo)^n(\Xot) = (\ad \Xt)^n(\Xot)
\]
($n \ge 0$). Among these basis elements we have the evident brackets
\[
[e_i, (\ad e_i)^n(e_{jk})] = (\ad e_i)^{n+1}(e_{jk})
\]
and all other (ordered) brackets vanish. 
\end{slm}

The proof spans segments \ref{2}-\ref{4}.

\subsection{}\label{2}
We first show that $S$ spans $L^{\PL}$. For this, by induction, it's enough to show that for any generator $Y$ and any basis element $Z$, the bracket $[Y,Z]$ is again a basis  element or zero. We check case by case by direct elementary calculation. If $Z = e_i$ ($i = 1$ or $2$) this is clear. Next we consider the case
\[
Z =  (\ad \Xo)^n(\Xoo).
\]
If $Y = e_1$ ok; if $Y = e_{jk} \in N$ then the bracket is zero. If $Y = e_2$ we show
\[
\tag{*}
[e_2,  (\ad \Xo)^n(\Xoo)] = 0
\]
by induction on $n$. The base case is a defining relation. For $n \ge 1$ we have
\begin{align*}
[e_2,  &(\ad \Xo)^n(\Xoo)]
	= [e_2, [\Xo, (\ad \Xo)^{n-1}(\Xoo)]] \\
	&= - [(\ad \Xo)^{n-1}(\Xoo), [\Xt, \Xo] ] 
	- [\Xo, [(\ad \Xo)^{n-1}(\Xoo), \Xt]] \\
	&= [\Xo,[\Xt,(\ad \Xo)^{n-1}(\Xoo)]] \\
	&= [\Xo, 0]
\end{align*}
by the inductive hypothesis as desired. The case 
\[
Z =  (\ad \Xt)^n(\Xtt)
\]
follows by symmetry. Finally, the case
\[
Z =  (\ad \Xo)^n(\Xot) =  (\ad \Xt)^n(\Xtt)
\]
is immediate. This completes the verification that $S$ spans $L^{\PL}$. 

\subsection{}\label{3}
We turn to the linear independence. We endow the free Lie algebra $L(\Gamma)$ with the $\ZZ^4$-grading in which $e_1$, $e_2$ are both homogeneous of degree $(1,0,0,0)$ and $e_{11},e_{12},e_{22}$ are homogeneous of degrees $(0,1,0,0),(0,0,1,0),(0,0,0,1)$, respectively. We use this alternate grading in place of the prevailing $\ZZ$-grading for the remainder of this paragraph. The relations \ref{1.1}(R) are homogeneous for this alternate grading. Since the Lie ideal $N$ is homogeneous, so is the Lie ideal $[N,N]$. Consequently, the alternate grading descends to a grading on $L^\m{PL}$.

There's an obvious homomorphism from $L^{\PL}$ to the free abelian Lie algebra on the 5 generators, hence the set of generators is linearly independent. Since the higher-degree elements
\[
(\ad \Xo)^n(\Xoo), \; (\ad \Xt)^n(\Xtt), \;
(\ad \Xo)^n(\Xot) = (\ad \Xt)^n(\Xot)
\]
are homogeneous of different degrees, it's enough to show that each one is nonzero. 

Let $L_{1}$ denote the polylogarithmic Lie algebra for $\thrpl$ with vector space basis
\[
Y_1, Y_{11}, (\ad Y_1)(Y_{11}), (\ad Y_1)^2(Y_{11}), \dots.
\]
The homomorphism
\[
L(\Gamma) \to L_1
\]
which maps
\[
e_1 \mapsto Y_1, 
\quad 
e_{11} \mapsto Y_{11},
\]
and the remaining generators to $0$, factors through $L^{\PL} \to L_1$. In $L_1$ the element 
$
(\ad Y_1)^n(Y_{11}) \neq 0
$, and it follows that
$
(\ad e_1)^n(e_{11}) \neq 0
$. 
Replacing 1's by 2's, we find that 
$
(\ad e_2)^n(e_{22}) \neq 0.
$
Sending instead
\[
e_1 \mapsto Y_1, 
\quad 
e_2 \mapsto Y_1, 
\quad 
e_{12} \mapsto Y_{11},
\]
and the remaining generators to $0$, we find that
$
(\ad e_1)^n(e_{12}) \neq 0.
$

\subsection{}\label{4}
Finally, the statement regarding the brackets among the basis elements is immediate from the defining relations, from the vanishing \ref{2}(*) and from its symmetrical twin
\[
[e_1,  (\ad \Xt)^n(\Xtt)] = 0.
\]
This completes the proof of Lemma \ref{liebasis}.

\section{The polylogarithmic Hopf algebra in abstraction}
\label{ab0}

\subsection{}\label{freepro}\label{ab1}
We fix a base field $k$ of characteristic $0$. We refer the reader to \S2 of \cite{mtmue} for an efficient review of the basics of free prounipotent groups in notation similar to ours. If $S$ is a set, we denote by $L(S)$ the free Lie algebra on $S$ and by $\nN(S)$ the free pronilpotent Lie algebra on $S$. We denote the associated \emph{free prounipotent $k$-group} by $\pi(S)$. We let $\Uu(S)$ denote the completed universal enveloping algebra of $\nN(S)$ and we let $A(S)$ denote the coordinate ring of $\pi(S)$. We recall that there's a nondegenerate pairing
\[
\langle \cdot, \cdot \rangle:
\Uu(S) \times A(S) \to k.
\]
We describe elements of the free monoid generated by $S$ as \emph{associative words} in order to distinguish them from Lie words. If $W$, $W'$ are associative words, we say that $W'$ is a \emph{subword} of $W$ if there exist words $W_1$, $W_2$ such that $W = W_1 W' W_2$.  If $\sum a_i w_i$ is a finite linear combination of associative words in $S$, we let $f_{\sum a_i w_i} = \sum a_i f_{w_i}$ denote the dual vector in $A(S)$.

\subsection{}\label{ab2}
Recall that $L^{\PL}$ denotes the polylogarithmic Lie algebra (\ref{1.1}). We repeat the notation of \S\ref{ab1} with the decoration `PL' everywhere; thus, $\nN^{\PL}$ denotes the pronilpotent completion, $\Uu^{\PL}$ the completed universal enveloping algebra, $\pi^{\PL}$ the prounipotent group, and $A^{\PL} = \Oo(\pi^{\PL})$ the associated Hopf algebra. In terms of the set $\Gamma$ of \S\ref{ap0}, there's a natural surjection
\[
\nN(\Gamma) \surj \nN^{\PL},
\]
hence an associated inclusion 
\[
A(\Gamma) \supset A^{\PL}
\]
of Hopf algebras. 

By a \emph{Lie word} $W \in L(\Gamma)$ in the set of generators $\Gamma$ we mean an element of the form
\[
[Y_1[Y_2[ \cdots [Y_{n-1},Y_n] \cdots]]]
\]
with $Y_i \in \Gamma$. We set
\[
\m{depth}(e_i) = 0,
\] 
\[
\m{depth}(e_{jk}) =1,
\]
and we define the depth of a Lie word to be the sum of the depths of its letters $Y_i$.

\label{ab2.5}
Let $\aA$ denote the kernel of 
\[
L(\Gamma) \surj L^{\PL}.
\]
The Lie ideal $\aA < L(\Gamma)$ is generated by all Lie words of depth $2$ together with the elements
\[
[e_1, \Xt], [\Xoo, \Xt], [\Xo, \Xtt], [e_1, \Xot]-[\Xt, \Xot].
\]
We refer to a set of free generators in a polynomial algebra as an \emph{algebra basis}.

\begin{slm}\label{ab3}
The elements
\[
\tag{*}
f_{e_1}, f_{e_2},
f_{\Xoo\Xo^n}, f_{\Xtt\Xt^n}, 
f_{\eot(\eo+\et)^n}
\quad
n \ge 0
\]
of $A(\Gamma)$ form an algebra basis of $A^{\PL}$.
\end{slm}

The proof of Lemma \ref{ab3} spans segments \ref{ab3s}-\ref{ab3e}.

\subsection{}\label{ab3s} 
We first show that the elements \ref{ab3}(*) are contained in $A^{\PL}$. 
Let $f$ be one of them. It's enough to show that $f$ vanishes on the two-sided ideal $\Uu \aA \Uu$. Thus, referring to \S\ref{ab2.5}, it's enough to show that for any of the Lie-word generators $\la \in \aA$ described there and any two associative words $W, W' \in \Uu$, 
\[
f(W\la W') = 0.
\]
This is clear for all but the two cases
\[
f_{\eot(\eo+\et)^n}(W[\eo,\et]W') 
\quad \mbox{and} \quad
f_{\eot(\eo+\et)^n}(W ([\eo,\eot]-[\et,\eot])W').
\]
In turn, some simple but tedious combinatorics show that these vanish as well. 

\subsection{}\label{rufprod}
The bijections
\[
(\Uu(\Gamma)/I^n)^\lor = \Uu(\Gamma)/I^n
\] 
induced by the basis of associative words in $\Gamma$ endow
\[
A(\Gamma) = \lim_\to (\Uu(\Gamma)/I^n)^\lor
\]
with a second product $\ast$ which satisfies
\[
f_v \ast f_w := f_{vw}.
\]
The operation $\ast$ also induces left and right actions of $A(\Gamma)$ on $A(\Gamma) \otimes A(\Gamma)$ in an obvious way.

\begin{slm}\label{ruffleprod1}
For any letter $e$ and any linear combination of associative words $\sum a_i w_i$, the deconcatenation coproduct in the shuffle algebra $A(\Gamma)$ satisfies 
\[
\Delta (f_{(\sum a_i w_i)e}) 
=  
(\Delta f_{\sum a_i w_i}) \ast f_e + f_{(\sum a_i w_i)e} \varotimes 1,
\]
and, symmetrically 
\[
\Delta (f_{e(\sum a_i w_i)}) 
=
f_e \ast  (\Delta f_{\sum a_i w_i})  + 1 \varotimes  f_{e(\sum a_i w_i)}.
\]
\end{slm}

\begin{slm}\label{coprod1}
The deconcatenation coproduct in the shuffle algebra $A(\Gamma)$ satisfies 
\[
\Delta f_{(e_1+e_2)^n} = 
\sum_{i+j = n} f_{(\eo+\et)^i} \varotimes f_{(\eo+\et)^j}.
\]
\end{slm}

The proofs of Lemmas \ref{ruffleprod1} and \ref{coprod1} are straightforward and we omit them. 

\subsection{}\label{ab3e}
Let $A''$ be the subalgebra of $A^{\PL}$ generated by the elements \ref{ab3}(*). Then $A''$ is closed under coproduct and is a Hopf subalgebra of $A^{\PL}$.
Combining lemmas \ref{ruffleprod1} and \ref{coprod1}, we have
\begin{align*}
\Delta f_{\eot(\eo+\et)^n} 
	&= f_{\eot} \ast \Delta f_{(\eo+\et)^n} + 1 \otimes f_{\eot(\eo+\et)^n} 
	\\
	&= f_{\eot} \ast 
	\sum_{i+j = n} f_{(\eo+\et)^i} \varotimes f_{(\eo+\et)^j}
	+ 1 \otimes f_{\eot(\eo+\et)^n} 
	\\
	&= \sum_{i+j = n} f_{\eot(\eo+\et)^i} \varotimes f_{(\eo+\et)^j}
	+ 1 \otimes f_{\eot(\eo+\et)^n} \in A'' \otimes A''.
\end{align*}
It follows that
\[
\phi: \Spec A^{\PL} \to \Spec A''
\]
is a surjection of graded prounipotent groups. Let $J'$, $J''$ denote the respective augmentation ideals, and consider the induced map of Lie coalgebras
\[
(L^{\PL})^\lor = J'/J'^2
\xfrom{\varphi}
J''/J''^2.
\]
Under $\varphi$, the images of the generators \ref{ab3}(*) map to the dual basis of the Lie algebra basis \ref{liebasis}. Hence it must be an isomorphism. Hence $A'' = A^{\PL}$. This completes the proof of Lemma \ref{ab3}.

\subsection{}
We record another way to finish the proof of Lemma \ref{ab3} after \S3.6, indicated by the referee. The functor corresponding to the group scheme $\pi^{\PL}$ is given by 

\[
R \mapsto \hat{L}^{\PL}_{R} 
:=
\varprojlim (L^{\PL} \otimes R) /(\Gamma^{i}L^{\PL} \otimes R)
\]
where $\Gamma^{i}L^{\PL}$ denotes the descending central series of $L^{\PL}$. By Lemma 2.3, any element $x$ of $\hat{L}^{\PL}_{R}$ can be written in a unique way as 
$$ x = a_{1}e_{1} + a_{2}e_{2} 
+ \sum_{n\geq 0} b_{n} \ad_{e_{1}}^{n}(e_{11}) 
+ \sum_{n\geq 0} c_{n} \ad_{e_{2}}^{n}(e_{22}) 
+ \sum_{n\geq 0} d_{n} \ad_{e_{1}}^{n}(e_{12}) $$
where $a_{i}, b_{i}, c_{i}, d_{i} \in R$. By \S3.6, $f_{e_{1}}$, $f_{e_{2}}$, $f_{e_{11}e_{1}^{n}}$ etc. factor through $\pi^{\PL}$ so they can be applied to this equality. This yields
$$  
a_{i} = f_{e_{i}}(x),
\text{ } 
b_{n} = (-1)^{n}f_{e_{11}e^{n}_{1}}(x),
$$
$$
c_{n} = 
(-1)^{n}
f_{e_{22}e_{2}^{n}}(x),
\text{ }
d_{n} = 
(-1)^{n}
f_{e_{12}
(e_{1}+e_{2})^{n}}(x).
$$
Thus the elements $f_{\lambda} \in \mathcal{O}(\pi^{\PL})$ form a complete set of coordinates on $\pi^{\PL}$.

\section{The geometric step}
\label{gstep}

\subsection{}\label{gs1}
Let $K(\tau, \upsilon, \si)$ denote the fraction field of the free prounipotent group on three generators $\tau, \upsilon, \si$ over $\QQ$. 
The elements
\[
f_\tau, f_\up, f_{\tau \up}, f_\si, f_{\tau\tau\up}, f_{\tau\up\up}, 
f_{\tau\si}, f_{\up\si}, f_{\tau\tau\tau\up}, f_{\tau\tau\up\up}, f_{\tau\up\up\up}
\]
(\ref{freepro}) 
are algebraically independent, and in what follows we may equivalently work over the sub-field of transcendence degree 11 generated by them. 

Letting
\[
\Gamma = \{e_1, e_2, e_{11}, e_{22}, e_{12}\}
\]
as above, we consider the set of elements of the noncommutative polynomial ring $\QQ\langle \Gamma \rangle$
\[
\Lambda_{\ge -\infty} = \bigcup_{i=1}^\infty \Lambda_{-i},
\mbox{ where }
\Lambda_1 = \Gamma,
\]
\[
\mbox{ and }
\Lambda_{-i} = \{ e_{11}e_1^{i-1}, e_{22}e_2^{i-1}, e_{12}(e_1+e_2)^{i-1}\}
\mbox{ for } i \ge 2.
\]
Since we'll be working primarily with $\Lambda_{\ge -4}$, we abbreviate
\[
\Lambda := \Lambda_{\ge -4}.
\]
We declare $\tau, \upsilon$ to have weight $-1$ and $\si$ to have weight $-3$; we declare the elements of $\Gamma$ to have weight $-1$. We consider the polynomial rings $\QQ[\{f_\la\}_\la]$ and $\QQ[\{\Phi^\rho_\la\}_{\la, \rho}]$ where $\la$ ranges over $\Lambda$ and $\rho$ ranges over the three generators $\{\tau, \upsilon, \si\}$ and is required to have weight equal to that of $\la$. Explicitly, the complete list of algebra generators of $\QQ[\{\Phi^\rho_\la\}_{\la, \rho}]$ is 
\begin{align*}
& \Phi^{\tau}_{e_1}, \Phi^{\tau}_{e_2}, \Phi^{\tau}_{e_{11}}, 
\Phi^{\tau}_{e_{22}}, \Phi^{\tau}_{e_{12}} 
\\
&\Phi^{\upsilon}_{e_1}, \Phi^{\upsilon}_{e_2}, \Phi^{\upsilon}_{e_{11}}, 
\Phi^{\upsilon}_{e_{22}}, \Phi^{\upsilon}_{e_{12}},
\\
& \Phi^\si_{e_{11}e_1e_1},
 \Phi^\si_{e_{22}e_2e_2},
  \Phi^\si_{e_{12}(e_1+e_2)^2},.
\end{align*}
 Thus, $\QQ[\{f_\la\}_\la]$ has Krull dimension $14$ whereas $\QQ[\{\Phi^\rho_\la\}_{\la, \rho}]$ has Krull dimension $13$. 
Define
\[
\theta: \Aa:=K(\tau, \upsilon, \si)[\{f_\la\}_\la] 
\to
\Ss:=K(\tau, \upsilon, \si)[\{\Phi^\rho_\la\}_{\rho,\la}]
\]
by
\begin{align*}
\theta(f_{e_1}) &= f_\tau \Phi^\tau_{e_1}+ f_\upsilon \Phi^\upsilon_{e_1} \\
\theta(f_{e_2}) &= f_\tau \Phi^\tau_{e_2}+ f_\upsilon \Phi^\upsilon_{e_2} \\
\theta(f_{e_{11}}) &= f_\tau \Phi^\tau_{e_{11}}+ f_\upsilon \Phi^\upsilon_{e_{11}} \\
\theta(f_{e_{22}}) &= f_\tau \Phi^\tau_{e_{22}}+ f_\upsilon \Phi^\upsilon_{e_{22}} \\
\theta(f_{e_{12}}) &= f_\tau \Phi^\tau_{e_{12}}+ f_\upsilon \Phi^\upsilon_{e_{12}} \\
\theta(f_{e_{11}e_1}) &= f_{\tau\tau} \Phi^\tau_{e_{11}} \Phi^\tau_{e_1} 
+ f_{\tau \upsilon}\Phi^\tau_{e_{11}} \Phi^\upsilon_{e_1}
+
 f_{\upsilon\tau} \Phi^\upsilon_{e_{11}} \Phi^\tau_{e_1} 
+ f_{\upsilon \upsilon}\Phi^\upsilon_{e_{11}} \Phi^\upsilon_{e_1}
\\
\theta(f_{e_{22}e_2}) &= \mbox{similar}
\\
\theta(f_{e_{12}(e_1+e_2)}) &= f_{\tau \tau} \Phi^\tau_{e_{12}} \Phi^\tau_{e_1+e_2} + \cdots
\\
\theta(f_{e_{11}e_1e_1}) 
&= f_{\tau\tau\tau} \Phi^\tau_{e_{11}} \Phi^\tau_{e_1} \Phi^\tau_{e_1} 
+ \cdots +
f_{\upsilon\upsilon\upsilon}
\Phi^\upsilon_{e_{11}} \Phi^\upsilon_{e_1} \Phi^\upsilon_{e_1} 
+ f_\si \Phi^\si_{e_{11}e_1e_1} 
\\
\theta(f_{e_{22}e_2e_2})  &= \mbox{similar}
\\
\theta(f_{e_{12}(e_1+e_2)^2})  &= \mbox{similar}
\\
\theta(f_{e_{11}e_1^3}) &= f_{\tau^4} \Phi^\tau_{e_{11}}(\Phi^\tau_{e_1})^3
+ \cdots + f_{\upsilon^4}\Phi^\upsilon_{e_{11}}(\Phi^\upsilon_{e_1})^3 
+f_{\si \tau} \Phi^\si_{e_{11}e_1e_1} \Phi^\tau_{e_1}
+f_{\si \upsilon} \Phi^\si_{e_{11}e_1e_1} \Phi^\upsilon_{e_1}
\\
\theta(f_{e_{22}e_2^3}) &= \mbox{similar}
\\
\theta(f_{e_{12}(e_1+e_2)^3}) &= \mbox{similar}
\end{align*} 
where
\[
\Phi^\tau_{e_1+e_2} = \Phi^\tau_{e_1}+ \Phi^\tau_{e_2}
\quad \mbox{and} \quad
\Phi^\upsilon_{e_1+e_2} = \Phi^\upsilon_{e_1}+ \Phi^\upsilon_{e_2}.
\]

\subsection{}\label{coocoo1}
Let $\pi(\tau, \upsilon, \si)$ be the graded free prounipotent group on three generators in weights $-1, -1, -3$. Let $\pi^\PL$ denote the prounipotent $\QQ$-group associated to the Lie algebra $L^\PL$ of \S\ref{1.1}. Let $A^\PL$ denote the coordinate ring of $\pi^\PL$. Let $A^\PL_{[\le n]}$ denote the subalgebra generated by elements in graded degree $\le n$. This is a Hopf-subalgebra and we let
\[
\pi^\PL_{\ge -n} = \Spec A^\PL_{[\le n]}
\]
be the associated quotient of $\pi^\PL$. Let
\[
{\bf Z}^1(\pi(\tau, \upsilon, \si), \pi^\PL_{\ge -4})^\Gm
\]
denote the functor from $\QQ$-algebras to sets sending
\[
R \mapsto 
Z^1(\pi(\tau, \upsilon, \si)_R, (\pi^\PL_{\ge -4})_R)^\Gm,
\]
the pointed set of $\Gm$-equivariant cocycles for the trivial group-action. We refer to an $R$-valued point of ${\bf Z}^1(\pi(\tau, \upsilon, \si), \pi^\PL_{\ge -4})^\Gm$ as an \emph{$R$-family of cocycles} or an \emph{$R$-cocycle} for short; we omit the repeating phrase ``for $R$ an arbitrary $\QQ$-algebra''. Let $\ev$ denote the map
\[
\tag{*}
\pi(\tau, \upsilon, \si) \times 
{\bf Z}^1(\pi(\tau, \upsilon, \si), \pi^\PL_{\ge -4})^\Gm
\to 
\pi(\tau, \upsilon, \si) \times \pi^\PL_{\ge -4}
\]
given on $R$-points by
\[
(\ga, c) \mapsto (\ga, c(\ga)).
\]
Let $\ev_K$ denote the map
\[
{\bf Z}^1(\pi(\tau, \upsilon, \si), 
\pi^\PL_{\ge -4})^\Gm_{K(\tau, \up, \si)}
\to 
(\pi^\PL_{\ge -4})_{K(\tau, \up, \si)}
\]
obtained from $\ev$ by base-change.

\begin{sprop}\label{coocoo4}
In the situation and the notation of segments \ref{gs1}-\ref{coocoo1}, there's a commuting square of functors  
\[
\xymatrix{
{\bf Z}^1(\pi(\tau, \upsilon, \si),
\pi^\PL_{\ge -4})^\Gm_{K(\tau, \up, \si)}
\ar[r]^-{\ev_K}
& 
(\pi^\PL_{\ge -4})_{K(\tau, \up, \si)}
\\
\Spec
K(\tau, \upsilon, \si)[\{\Phi^\rho_\la\}_{\rho,\la}]
\ar[r]_-{\Spec \theta} \ar[u]_-\sim
&
\Spec
 K(\tau, \upsilon, \si)[\{f_\la\}_\la]
  \ar[u]_-\sim
}
\]
in which the vertical maps are isomorphisms.
\end{sprop}

The proof of proposition \ref{coocoo4} spans segments \ref{cc4start}-\ref{cc4end}.

\subsection{}\label{cc4start}
Let $U$ be a prounipotent $\QQ$-group with  coordinate ring $A$ and let $\Uu$ be the completed universal enveloping algebra of $\Lie U$. Then there's an isomorphism between $A$ and the topological dual $\Uu^\lor$. Given $f\in A$ and $w \in \Uu$ we denote the action of the linear functional associated to $f$ on $w$ by
\[
\langle f, w \rangle.
\] 
If $\la$ is a linear combination of elements of $ \Lambda$ and $w$ is an associative word in the alphabet $\{\tau, \up, \si\}$, we define 
\[
\phi^w_\la: 
{\bf Z}^1(\pi(\tau, \upsilon, \si), 
\pi^\PL_{\ge -4})^\Gm
\to
\AA^1_\QQ
\]
by the formula
\[
\phi^w_\la(c) = \langle c^\sharp f_\la, w \rangle
\]
on $R$-valued points. Note that for any $\Gm$-equivariant $R$-cocycle $c$ and $\la \in \Lambda$,
\[
c^\sharp f_\la = \sum_w \phi^w_\la(c) f_w
\]
is homogeneous of weight equal to the weight of $\la$. Thus, $\phi^w_\la = 0$ unless $\wt(w) = \wt(\la)$. Note also that $\phi^w_\la$ is linear in the subscript, that is
\[
\phi^w_{\sum_i a_i \la_i} = \sum_i a_i \phi^w_{\la_i}.
\]

For future use, we formulate and prove the following proposition in the slightly more general setting of a free graded prounipotent group $\pi(\Si)$ on the \textit{graded set} (i.e. $\ZZ$-indexed family of sets)
\[
\Si = \Si_{-1} \cup \Si_{-2} \cup \Si_{-3} \cup \cdots
\cup \Si_{-n}
\]
with $\Si_{-1}$ an arbitrary finite set, $\Si_{i} = 0$ for $i$ even $\le -2$, $\Si_{i} = \{\si_{-i}\}$ of size one for $i$ odd $\le -3$, and with $\pi^\PL$ in place of $\pi^\PL_{\ge -4}$.

\begin{sprop}\label{cccoeffs}
We continue with the situation and the notation of segments \ref{coocoo1}, \ref{cc4start}. Let 
\[
c: \pi(\Si)_R \to \pi^\PL_R
\]
be a $\Gm$-equivariant $R$-cocycle. Then for $0 \le r \le n$, $\tau_1, \dots, \tau_r \in \Si_{-1}$ and $\si \in \Si_{r-n}$, we have
\begin{align}
\phi^{\si \tau_1 \cdots \tau_r}_{\eoo \eo^{n-1}} (c) 
&= 
\phi^\si_{\eoo \eo^{n-r-1}}(c) \phi^{\tau_1}_{\eo}(c) 
\cdots \phi^{\tau_r}_{\eo}(c)
\\
\phi^{\si \tau_1 \cdots \tau_r}_{\ett \et^{n-1}} (c) 
&= 
\phi^\si_{\ett \et^{n-r-1}}(c) \phi^{\tau_1}_{\et}(c) 
\cdots \phi^{\tau_r}_{\et}(c)
\\
\phi^{\si \tau_1 \cdots \tau_r}_{\eot (\eo+\et)^{n-1}} (c) 
&= 
\phi^\si_{\eot (\eo+\et)^{n-r-1}}(c) \phi^{\tau_1}_{\eo+\et}(c) 
\cdots \phi^{\tau_r}_{\eo+\et}(c),
\end{align}
and if $w$ is any word not occurring in the above equations then
\[
\tag{4}
\phi^{w}_\la(c) = 0.
\] 
Conversely, given arbitrary elements $a^\rho_\la \in R$ for $\rho \in \Si$ and $\la \in \Lambda_{ \ge -\infty}$ of equal half-weight, there exists one and only one $\Gm$-equivariant $R$-cocycle $c$ satisfying 
\[
\phi^\rho_\la(c) = a^\rho_\la.
\]
\end{sprop}

Proposition \ref{cccoeffs} and its proof are similar to Proposition 3.10 of \cite{PolGonI}. The proof spans segments \ref{costart}-\ref{coend}. 

\subsection{}\label{costart}
We begin with a formal calculation, in which $\Si_{-1}$ may be an arbitrary finite set, and $\{a^\tau\}_{\tau \in \Si_{-1}}$ a family of commuting coefficients. In this abstract setting, we claim that
\[
\left( \sum_{\tau \in \Si_{-1}} a^\tau f_\tau \right) ^{\Sha n}
=
n! \sum_{\tau_1, \dots, \tau_n \in \Si_{-1}}
 a^{\tau_1} \cdots a^{\tau_n} f_{\tau_1 \cdots \tau_n}.
\]
Indeed, the left side of the equation
\begin{align*}
	&= \sum_{\tau_1, \dots, \tau_n} (a^{\tau_1}f_{\tau_1}) \mbox{\tiny{$\sha$}}
		\cdots \mbox{\tiny{$\sha$}} (a^{\tau_n}f_{\tau_n}) \\
	&= \sum_{\tau_1, \dots, \tau_n} a^{\tau_1}\cdots a^{\tau_n}
		\left( 
			\underset{\mbox{\tiny {of} }\tau_1, \dots, \tau_n}
				{\sum_{\mbox{\tiny{permutations} }p}}
			f_{\tau_1^p \cdots \tau_n^p}
		\right) \\
	&= \sum_p 
		\underbrace{
		\sum_{\tau_1, \dots, \tau_n} a^{\tau_1}
		\cdots a^{\tau_n} f_{\tau_1^p \cdots \tau_n^p}
			}_{\mbox{\tiny{independent of} }p}, \\
\end{align*}
which equals the right side of the equation.

\subsection{}\label{cc6}

Returning to our concrete situation, we focus on equation \ref{cccoeffs}(3) and, simultaneously, on the case $\la = \eot(\eo+\et)^{n-1}$ of equation \ref{cccoeffs}(4). We have, tautologically
\[
\tag{A}
c^\sharp f_{\eot(\eo+\et)^n} = 
\sum_{\set{w}{\wt w = n+1}}
\phi^w_{\eot(\eo+\et)^n}(c)f_w.
\]
We wish to compute the coproduct $\Delta$ of both sides, remembering that $c^\sharp$, since it corresponds to a cocycle for the trivial group action, preserves the coproduct. On the right side, we have 
\[
\Delta({\rm RHS}) = \sum_{\set{w', w''}{\wt(w')+\wt(w'') = n+1}}
\phi^{w'w''}_{\eot(\eo+\et)^n}(c)f_{w'} \otimes f_{w''}.
\]
On the left we have
\begin{align*}
\Delta({\rm LHS}) =
(c^\sharp \otimes c^\sharp) \left(
\sum_{i+j = n} f_{\eot(\eo+\et)^i }
\otimes \frac{f_{\eo+\et}^{\Sha j}}{j!}
+ 1 \otimes f_{\eot(\eo+\et)^n}
\right)
\intertext{by Lemmas \ref{ruffleprod1} and \ref{coprod1},}
\\
= \sum_{i+j = n}
\left(
\sum_u \phi^u_{\eot(\eo+\et)^i}(c)f_u
\right) 
\otimes 
\frac{
\left( 
\sum_{\tau \in \Si_{-1}}
\phi^\tau_{\eo+\et}(c)f_\tau
\right)^{\Sha j}
}{j!}
\\
+ 1 \otimes \sum_v \phi^v_{\eot(\eo+\et)^n}(c) f_v
\\
= \underset{\tau_1, \dots, \tau_j} {\underset{u} {\underset{i+j = n} \sum} }
\phi^u_{\eot(\eo+\et)^i}(c)
\phi_{\eo+\et}^{\tau_1}(c) \cdots \phi_{\eo+\et}^{\tau_j}(c)
f_u \otimes f_{\tau_1 \cdots \tau_j}
\\
+ \sum_v \phi^v_{\eot(\eo+\et)^n}(c) 1\otimes f_v
\end{align*}
by \S\ref{costart}. Taking the coefficient of $f_v \otimes f_\tau$ with $\tau \in \Si_{-1}$ and $v$ an arbitrary word of length $n \ge 1$, we obtain
\[
\tag{B}
\phi
^{v\tau}
_{e_{12}(e_1+e_2)^n}
(c)
=
\phi^v
_{e_{12}(e_1+e_2)^{n-1}}
(c)
\phi^\tau_{e_1+e_2}(c)
\]
while taking the coefficient of $f_v \otimes f_\si$ with $\si\in \Si_i$ for $i<-1$ and $v$ an arbitrary word of length $n+1-i \ge 1$, we obtain
\[
\tag{C}
\phi^{v\si}
_{e_{12}(e_1+e_2)^n}
(c)=0.
\]
Equations \ref{cccoeffs}(3) and (4)($\la = \eot(\eo+\et)^{n-1}$) follow.

\subsection{}\label{coend}
We turn to the second clause of the proposition. Equations \ref{cc6}(B) and (C) show explicitly how all coefficients of $c^\sharp f_\la$ in equation \ref{cc6}(A) are determined by the values $\phi^\rho_\la$ for $\rho \in \Si$ an individual letter. Alternatively, this is just an immediate consequence of the fact that the elements $f_\la$ for $\la \in \Lambda_{\ge -\infty}$ form an algebra basis of $\pi^\PL$ on the one hand, and that the elements $\rho \in \Si$ form a free set of generators for the prounipotent group $\pi(\Si)$ on the other. The fact that the generators are not elements of the group may cause some confusion, so we take the time to spell out the existence: given elements $a^\rho_\la \in R$ as in the proposition, we define
\[
c: \pi(\Si)_R \to \pi^\PL_R
\]
to be the unique homomorphism
such that 
\[
\langle
(\Lie c)(\rho), f_\la
\rangle
= a^\rho_\la
\]
whenever the half-weights of $\rho$ and $\la$ are equal, and $=0$ otherwise. Here, as usual, we regard $(\Lie c)(\rho)$ as belonging to the universal enveloping algebra of $\pi^\PL$, and the bracket refers to the natural pairing between the coordinate ring and the universal enveloping algebra. This completes the proof of Proposition \ref{cccoeffs}.

\subsection{}\label{cc4end}

We now prove Proposition  \ref{coocoo4}. The second clause of Proposition \ref{cccoeffs} says that the map
\[
\Spec \QQ
[\{\Phi^\rho_\la\}_{\rho, \la}]
\from
{\bf Z}^1(\pi(\tau, \upsilon, \si),
\pi^\PL_{\ge -4})^\Gm
\]
induced by the functions 
\[
\AA^1_\QQ
\xfrom{\Phi^\rho_\la}
{\bf Z}^1(\pi(\tau, \upsilon, \si),
\pi^\PL_{\ge -4})^\Gm
\]
is an isomorphism of $\QQ$-schemes. Lemma \ref{ab3} implies that the map
\[
\Spec \QQ[\{ f_\la\}_{\wt \la \ge -4} ]
\from
\pi^\PL_{\ge -4}
\]
induced by the functions $f_\la$ is also an isomorphism of $\QQ$-schemes. 

Let $A(\tau, \upsilon, \si)$ denote the coordinate ring of $\pi(\tau, \upsilon, \si)$. In terms of the functions $\Phi^\rho_\la$ on ${\bf Z}^1(\pi(\tau, \upsilon, \si),
\pi^\PL_{\ge -4})^\Gm$, the functions $f_\la$ on $\pi^\PL_{\ge -4}$ and the functions $f_w$ on $\pi(\tau, \upsilon, \si)$, the universal cocycle evaluation map
$\ev$ of \S\ref{coocoo1}(*) is computed as follows. Since the map $\ev$ commutes with the first projection, we have
\[
\ev^\sharp(f_w) = f_w
\]
for all words $w$ in the generators $\{\tau, \upsilon, \si\}$. 
Let $Q$ be the set of words which occur in equations 1-3 of Proposition \ref{cccoeffs} and let $Q_n \subset Q$ denote the words of half-weight $n$. We use the same equations to define new functions $\Phi^w_\la$ on the space of equivariant cocycles (replace each lower-case $\phi$ by an upper-case $\Phi$). We then have for $(\ga, c)$ an $R$-valued point of 
\[
\pi(\tau, \upsilon, \sigma)
\times
{\bf Z}^1
(\pi(\tau, \upsilon, \sigma),
\pi_{\ge -4}^\PL
)^\Gm
\]
and $\la \in \Lambda_n$, 
\begin{align*}
    (\ev^\sharp f_\la)(\ga, c)
    &= (c^\sharp f_\la) (\ga) 
    \\
    &= \Big(\sum_{\wt w = n}
    \phi^w_\la(c) f_w \Big) (\ga)
    \\
    &= \Big( \sum_{w \in Q_n}
    f_w \Phi^w_\la \Big)(\ga, c)
\end{align*}
where the last equation holds by Proposition \ref{cccoeffs}. Hence we have
\[
\ev^\sharp f_\la
=
\sum_{w \in Q_n} f_w \Phi^w_\la.
\]
These equations are written out explicitly in \S\ref{gs1}. 
This completes the proof of Proposition \ref{coocoo4}.


We now turn to the problem of constructing a nonzero element in the kernel of $\theta$.

\subsection{}
\label{6s1}
Consider the following element of $\Aa[X,Y]$: 
\begin{multline*}
Q_1(X,Y) = - \Big[ (f_{\tau\upsilon} f_{\upsilon} - f_{\tau} f_{\upsilon\upsilon}) X 
+ ( f_{\tau\tau}f_{\upsilon} - f_{\tau} f_{\upsilon\upsilon}) Y\Big] f_{e_{11}e_{1}^{3}} +
\\ 
\Big[f_{\upsilon} f_{e_{11}e_{1}} - (f_{\upsilon\tau} Y + f_{\upsilon\upsilon}X) f_{e_{11}} \Big](f_{\tau^{4}} Y^{3} + f_{\tau\upsilon^{3}} X^{3}
+ (f_{\tau(\tau^{2}\sh \upsilon)}) XY^{2} 
+ (f_{\tau(\upsilon^{2}\sh \tau)}) YX^{2})
\\ 
+ 
\Big[-f_{\tau} f_{e_{11}e_{1}} + (f_{\tau\upsilon} X + f_{\tau\tau}Y) f_{e_{11}}\Big]
(f_{\upsilon\tau^{3}} Y^{3} + f_{\upsilon^{4}} X^{3}
+ (f_{\upsilon(\tau^{2}\sh \upsilon)}) XY^{2} 
+ (f_{\upsilon(\upsilon^{2}\sh \tau)}) YX^{2})
+
\\
(f_{\sigma\tau} Y +  f_{\sigma\upsilon} X)
\frac{1}{f_{\sigma}} \Big[ f_{e_{11}e_{1}^{2}} - 
\Big[f_{\upsilon} f_{e_{11}e_{1}} - (f_{\upsilon\tau}Y + f_{\upsilon\upsilon}X) f_{e_{11}}  \Big]\
(f_{\tau\tau\tau}Y^{2} + f_{\tau\upsilon\upsilon}X^{2} 
+ f_{\tau(\tau\sh\upsilon)}XY)
\\
+ \Big[-f_{\tau}f_{e_{11}e_{1}} + (f_{\tau\upsilon} X + f_{\tau\tau}Y) f_{e_{11}}\Big]
(f_{\upsilon\tau\tau}Y^{2} + f_{\upsilon\upsilon\upsilon}X^{2} 
+ f_{\upsilon(\tau\sh\upsilon)}XY) \Big]
\end{multline*}
and let $a_{i,j}$ denote the coefficient of the monomial $X^iY^j$.
We use the elements $a_{i,j} \in \Aa$ to construct a collection of elements of $\Aa[Y]$: we define 
\[
A_{1}(Y) =  - (f_{\tau\upsilon} - \frac{1}{2} f_{\tau}f_{\upsilon}) \Big[  f_{\upsilon\upsilon}f_{e_{11}} \Big]
\]
\begin{multline*} B_{1}(Y) = (f_{\tau\upsilon} - \frac{1}{2} f_{\tau}f_{\upsilon}) \Big[ f_{\upsilon} f_{e_{11}e_{1}} - f_{\upsilon\tau}Y f_{e_{11}} \Big]
\\ + (f_{\upsilon\tau} - \frac{1}{2} f_{\upsilon}f_{\tau}) \Big[
f_{\tau\upsilon} f_{e_{11}} \Big] Y - 
\Big[ (f_{\tau\upsilon} f_{\upsilon} - f_{\tau} f_{\upsilon\upsilon}) \Big] \Big[ f_{e_{11}e_{1}} - \frac{1}{2} f_{e_{11}} f_{e_{1}} \Big] 
\end{multline*}
\[
C_{1}(Y) = (f_{\upsilon\tau} - \frac{1}{2} f_{\upsilon}f_{\tau}) \Big[  -f_{\tau}f_{e_{11}e_{1}} + f_{\tau\tau} Y f_{e_{11}} \Big] Y - ( f_{\tau\tau}f_{\upsilon} - f_{\tau} f_{\upsilon\upsilon}) Y
\Big[ f_{e_{11}e_{1}} - \frac{1}{2} f_{e_{11}}f_{e_{1}} \Big]
\]
\[
{Q}_{1,a}(Y) = \sum_{i=0}^{4}\sum_{j=0}^{4} a_{ij} (2A_1)^{4-i} Y^{j} \sum_{k'=0}^{\lfloor \frac{i}{2} \rfloor} {i \choose 2k'}  (B_1^{2}  - 4A_1C_1)^{k'} (-B_1)^{i-2k'}
\]
\[
{Q}_{1,b}(Y) = \sum_{i=0}^{4}\sum_{j=0}^{4} a_{ij} (2A_1)^{4-i} Y^{j}  \sum_{k'=0}^{\lfloor \frac{i-1}{2} \rfloor } {i \choose 2k'+1}  (B_1^{2} - 4A_1C_1)^{k'} (-B_1)^{i-2k'-1}
\]
\begin{equation}
P_{1}(Y) = {Q}_{1,a}^{2} - (B_1^{2} - 4A_1C_1) {Q}_{1,b}^{2}.
\end{equation}
This last polynomial $P_1(Y)$ is divisible by a power of $Y$ and we let $p_1(Y)$ be the result of dividing by this factor. Polynomials $p_{2}(Y), p_{3}(Y) \in \Aa[Y]$ are defined similarly with $(e_{11},e_{1})$ replaced respectively by $(e_{22},e_{2})$ and $(e_{12},e_{1}+e_{2})$. Finally, we define $F \in \Aa$ to be the double-resultant
\[
\tag{5}
F = Res_Y\Big( Res_X \big( p_{1}(X), p_{2}(Y-X)  \big), p_{3}(Y) \Big).
\]

\sProposition{6s2}{
The element $F$ of $\Aa$ defined above is a nonzero element of the kernel of $\theta: \Aa \to \Ss$.
}

\begin{proof}
Needless to say, given a sufficiently powerful computer this could be easily checked via direct computation. Indeed, we do rely on computer verification for the claim that $F \neq 0$ \cite{M05code}. We nevertheless give a somewhat more nuanced account of the second claim, highlighting those steps in the construction which require something more than direct manipulation. 

We denote by $Q_1^\theta(X,Y)$ the image of $Q_1(X,Y)$ in $\Ss[X,Y]$, and similarly for $A_1^\theta(Y), B_1^\theta(Y), C_1^\theta(Y)$. Direct manipulation shows on the one hand that
\[
\tag{Q}
Q_1^\theta(\Phi^\up_{e_1}, \Phi^\tau_{e_1}) = 0,
\]
and on the other hand that 
\[
\tag{ABC}
A^\theta_1(\Phi^\tau_{e_1}) \cdot
 (\Phi^\up_{e_1})^2 + B^\theta_1(\Phi^\tau_{e_1}) \cdot \Phi^v_{e_1} + C^\theta_1(\Phi^\tau_{e_1}) = 0.
\]
Let
\[
\Delta_1(Y) = B_1(Y)^2 - 4 A_1(Y)C_1(Y),
\]
denote by $\Aa[Y,\de_1]$ the $\Aa[Y]$-algebra
\[
\Aa[Y,\de_1] := \Aa[Y, t]/(t^2 - \Delta_1(Y))
\]
and denote by $\de_1$ the equivalence class of $t$; denote by $\Ss[Y,\de_1^\theta]$ the $\Ss[Y]$-algebra
\[
\Ss[Y,\de^\theta_1] := \Ss[Y, u]/(u^2 - \Delta^\theta_1(Y)),
\]
and let $\de_1^\theta$ denote the equivalence class of $u$.   We denote the induced homomorphism 
\[
\Ss[Y, \de_1^\theta] \from 
\Aa[Y,\de_1]
\]
as well as the homomorphism of localizations
\[
\Ss[Y, \de_1^\theta, {A_1^\theta}^{-1}] \from 
\Aa[Y,\de_1, A_1^{-1}]
\]
simply by $\theta$. We sometimes write `$\de_1(Y)$' in place of `$\de_1$' in order to emphasize that it's contained in an algebra over a polynomial algebra in $Y$ and can be specialized to particular values of $Y$ in any $\Aa$-algebra. We then denote by $\de_1^\theta(?)$ the specialization of $\de_1^\theta$ at $Y = ?$. Similarly, we sometimes write `$\de_1^\theta(Y)$' and `$\de_1^\theta(?)$'. In this notation, we have the equation
\[
\left( 
\Phi^\up_{e_{1}} -  
\frac{-B^\theta_1(\Phi^\tau_{e_1}) + \de_1^\theta(\Phi^\tau_{e_1})}{2 A_1^\theta(\Phi^\tau_{e_1})}
\right)
\left( 
\Phi^\up_{e_{1}} -  
\frac{-B_1^\theta(\Phi^\tau_{e_1}) - \de_1^\theta(\Phi^\tau_{e_1})}{2 A_1^\theta(\Phi^\tau_{e_1})}
\right) = 0
\]
in the ring
\[
\Ss[\delta_1^\theta(\Phi^\tau_{e_1}), {A_1^\theta}(\Phi^\tau_{e_1})\inv]
=
\Ss_{A_1^\theta(\Phi^\tau_{e_1})}[s]/(s^2 - \Delta^\theta_1(\Phi^\tau_{e_1}))
\]
where the subscript denotes localization and $\de_1^\theta(\Phi^\tau_{e_1})$ corresponds to $s$.

This ring is integral. This follows from the following general fact. If $R$ is an integral domain with function field $K$ and $f \in R[x]$ is monic and irreducible over $K$, then $K[t]/(f)$ is again integral. Indeed,
\[
R[t]/(f) \from R
\]
is flat, so $K[t]/(f) \from R[t]/(f)$ is obtained from an injective map via flat base-change. Since $K[t]/(f)$ is a field, it follows that $R[t]/(f)$ is integral.

Consequently, we have 
\[
\Phi^\up_{e_{1}} 
=
\frac{-B^\theta_1(\Phi^\tau_{e_1}) + \ep \de_1^\theta(\Phi^\tau_{e_1})}
{2 A_1^\theta(\Phi^\tau_{e_1})}
\]
in the ring
\[
\Ss[\delta_1^\theta(\Phi^\tau_{e_1}), {A_1^\theta}(\Phi^\tau_{e_1})\inv]
\]
for some $\ep \in \{1, -1\}$.

Direct calculation in the ring $\Aa[Y,\de_1(Y), A_1(Y)\inv]$ shows that we have
\[
Q_1 \left(
\frac{-B_1(Y) +\ep \de_1(Y)}{2 A_1(Y)},
Y \right)
=
\frac{Q_{1,a}(Y) + \ep \de_1(Y) Q_{1,b}(Y)}
{(2A_1(Y))^4}.
\]
It follows that in $\Ss[A^\theta_1(\Phi^\tau_{e_1})\inv]$ we have 
\begin{align*}
P_1^\theta(\Phi^\tau_{e_1}) 
= \big(
Q_{1,a}^\theta(\Phi^\tau_{e_1}) 
- \ep \de_1^\theta(\Phi^\tau_{e_1}) Q_{1,b}^\theta(\Phi^\tau_{e_1}) 
\big)
\\
\cdot \big(
Q_{1,a}^\theta(\Phi^\tau_{e_1}) + \ep \de_1^\theta(\Phi^\tau_{e_1}) Q_{1,b}^\theta(\Phi^\tau_{e_1}) 
\big)
\\
=
(2A_1^\theta(\Phi^\tau_{e_1}))^4
\big(
Q_{1,a}^\theta(\Phi^\tau_{e_1}) - \ep \de_1^\theta(\Phi^\tau_{e_1}) Q_{1,b}^\theta(\Phi^\tau_{e_1}) 
\big)
\\
\cdot Q_1^\theta \left(
\frac{-B_1(\Phi^\tau_{e_1}) +\ep \de_1(\Phi^\tau_{e_1})}{2 A_1(\Phi^\tau_{e_1})},
\Phi^\tau_{e_1} \right)
\\
= 
(2A_1^\theta(\Phi^\tau_{e_1}))^4
\big(
Q_{1,a}^\theta(\Phi^\tau_{e_1}) - \ep \de_1^\theta(\Phi^\tau_{e_1}) Q_{1,b}^\theta(\Phi^\tau_{e_1}) 
\big)
\\
\cdot Q^\theta_1 (\Phi^\up_{e_1},\Phi^\tau_{e_1})
\\
=0.
\end{align*}
Since the localization map
$
\Ss \to \Ss[A^\theta_1(\Phi^\tau_{e_1})\inv]
$
is injective, it follows that
\[
P_1^\theta(\Phi^\tau_{e_1}) = 0
\]
already in $\Ss$. Since $\Ss$ is integral, it follows also that 
\[
p_1^\theta(\Phi^\tau_{e_1}) = 0
\]
Similarly, we have
$
p_2^\theta(\Phi^\tau_{e_2})=0,
$
and
\[
p_3^\theta(\Phi^\tau_{e_1} + \Phi^\tau_{e_2})=0.
\]
Consequently, 
$
\theta(F) = 0,
$
as claimed. 
\end{proof}

\section{The arithmetic step}
\label{sobsection}

For $Z$ an open subscheme of $\Spec(\mathbb{Z})$, we let
\[
\pi_1^\m{MT}(Z)^\times = \Gm \ltimes \pi_1^\un(Z)^\times
\]
denote the \emph{lexicographic} mixed Tate Galois group of $Z$. We will focus primarily on our special case $Z= \Spec(\mathbb{Z}[1/6])$, yet along the way will have occasion to make statements that hold equally for arbitrary $Z$. We also let $X = \thrpl$ and we let $\pi_1^\un(X, \vec{1_0})^\times$ denote the \emph{lexicographic} unipotent fundamental group of $X$ at the tangent vector $1$ based at the missing point $0$ \cite[\S4]{DelGon}. We let $\pi^\PL(X)^\times$ denote its polylogarithmic quotient.  We let $\Uu(Z)^\times$ denote the completed universal enveloping algebra of $\pi_1^\un(Z)^\times$. In this section we recall from \cite{PolGonI} the construction of generators 
$\tau_2, \tau_3, \si \in \Uu(\ZZ[1/6])^\times$ 
and write the ensuing shuffle coordinates as polynomials in unipotent motivic $n$-logarithms. We find it helpful to have several different notations available: we denote the generators of half-weight $-1$ by $\tau_2, \tau_3$ when we wish to emphasize the associated primes, by $\tau, \up$ when we wish instead to lighten the notation, and simply by $2,3$ when we wish to lighten notation while nevertheless emphasizing the associated primes (especially when words in the generators occur as subscripts). 

\subsection{}
\label{lidef}
Recall that the de Rham realization $\pi_1^\un(X, 1_0)^{\times,\dR}$ of the unipotent fundamental group of $X$ is free prounipotent on two generators, determined by the choice of 1-forms
\[
\frac{dt}{t}, \quad \frac{dt}{1-t}
\]
which define a basis of $H_\dR^1(X_\QQ)$. We here denote the corresponding generators by 
\[
d_0, \quad d_1.
\]
Moreover, the torsor $\pi_1^\un(X, 1_0, a)^{\times, \dR}$ of unipotent de Rham paths from $1_0$ to $a$ is canonically trivialized by a special ``de Rham'' path which we denote by $p^\dR$. In our lexicographic ordering, the motivic polylogarithm $\Liu_n(a)$ for $a \in X(Z)$ and $n \ge 1$ is defined to be the composite 
\[
\pi_1^\un(Z)^\times
 \xto{o(p^\dR)} 
 \pi_1^\un(X, 1_0, a)^\times 
\xto{f_{d_0^{n-1}d_1}}
\AA^1_\QQ
\]
where $o(p^\dR)$ denotes the orbit map
\[
\gamma \mapsto \gamma p^\dR.
\]
The motivic logarithm $\log^{\frak{u}}(a)$ is defined similarly with $f_{d_{0}}$ in place of $f_{d_0^{n-1}d_1}$.
 
\subsection{}\label{cob1}
We let $A(Z)^\times = \Oo(\pi_1^\un(Z)^\times)$. Recall that there's a canonical isomorphism of $\QQ$-vector spaces
\[
A(Z)^\times_n = \Uu(Z)^{\times,\lor}_{-n}.
\]
Recall that
$
A(Z)^\times_1 
$
has basis $\log^\uU(q)$ for $q \notin Z$. From now on we take $Z = \Spec \ZZ[1/6]$. In this case a basis of $A(Z)^\times_1$ is given by $\logu(2), \logu(3)$. We define $\tau_q \in \Uu(Z)^\times_{-1}$ to be the dual of $\logu(q)$ with respect to this basis. 

\subsection{}\label{cob2}
We let $E(Z)_n \subset A(Z)^\times_n$ denote the space of extensions and we let $D(Z)_n \subset A(Z)^\times_n$ denote the space of decomposables. According to Proposition 4.7 of \cite{PolGonI}, the elements $\Liu_3(-2)$, $\Liu_3(3)$ span a subspace $P(Z)_3$ of $A(\ZZ[1/6])^\times_3$ complementary to $E(Z)_3+D(Z)_3$. Based on this arbitrary choice, we let $\si \in \Uu(\ZZ[1/6])^\times_{-3}$ be the unique element which pairs with $P(Z)_3+D(Z)_3$ to zero and pairs with $\zeu(3)$ to $1$. 

\subsection{}\label{cob3}
The $\QQ$-vector space $\Uu(\ZZ[1/6])^\times_{-i}$ for $i = 1,2,3,4$ has a vector space basis consisting of associative words of half-weight $-i$ in the generators $\{\tau_2, \tau_3, \si\}$. If $w$ is such a word, we denote by $f_w \in A(\ZZ[1/6])^\times_i$ (as above) the function dual to $w$ with respect to this basis. The choice of ordering $\tau_2 < \tau_3 < \si$ gives rise to a set of Lyndon words whose duals
\begin{align*}
& f_{\tau_{2}}, f_{\tau_{3}} \\
& f_{\tau_{2} \tau_{3}} \\
& f_\si, f_{\tau_{2} \tau_{2} \tau_{3}}, f_{\tau_{2} \tau_{3}\tau_{3}} \\
& f_{\tau_{2} \si}, f_{\up \si}, f_{\tau_{2} \tau_{2}\tau_{2} \tau_{3}},
	 f_{\tau_{2}\tau_{2}\tau_{3}\tau_{3}}, f_{\tau_{2}\tau_{3}\tau_{3}\tau_{3}}
\end{align*}
form an algebra basis of the subalgebra $A(Z)^\times_{[\le 4]}$ of $A(Z)^\times$ generated by elements of degree $\le 4$. We refer to these as \emph{shuffle coordinates} on $\pi_1^\un(Z)^\times_{\ge -4}$.

\subsection{}\label{}
Set
\begin{align*}
\Ee_1 = \{\logu (2), \logu (3)\} && \Pp_1 = \emptyset \\
\Ee_2 = \emptyset && \Pp_2 = \{ \Liu_2 (-2) \} \\
\Ee_3 = \{\zeu(3)\} && \Pp_3 = \{\Liu_3(-2), \Liu_3(3) \} \\
\Ee = \bigcup_{i = 1}^3 \Ee_i && \Pp = \bigcup_{i = 1}^3 \Pp_i.
\end{align*}
Then $\Ee \cup \Pp$ forms a second algebra basis of $A(Z)^\times_{[\le 3]}$. We refer to its elements as \emph{polylogarithmic coordinates}.

\subsection{Remarks concerning functoriality}\label{cob4}
We wish to import computations carried out for $A(\Spec \ZZ[1/2])^\times$ to $A(\Spec \ZZ[1/6])^\times$. For this purpose, we temporarily allow $Z$ to vary among the open subschemes of $\Spec \ZZ$. The structures discussed above ($\pi_1^\un(Z)^\times$, $\Uu(Z)^\times_{-n}$, $A(Z)^\times_{n}$, $E(Z)_n$, $D(Z)_n$) are functorial with respect to $Z$. An inclusion $\iota: Z' \subset Z$ of open subschemes of $\Spec \ZZ$ (corresponding to an inclusion of finite sets of primes $S' \supset S$) gives rise to a surjection
\[
\tag{*}
\iota_*:\pi_1^\un(Z')^\times \surj \pi_1^\un(Z)^\times
\]
and an injection
\[
\iota^\sharp: A(Z')^\ti \supset A(Z)^\ti.
\]
In terms of any choice of homogeneous free generators of $\pi_1^\un(Z')^\ti$, (*) corresponds to the quotient by the normal subgroup generated by $\tau_q$ for $q \in Z \setminus Z'$, and so, $A(Z)^\ti$ is the corresponding shuffle subalgebra. In particular, a set of generators $\Si'$ for $\pi_1^\un(Z')^\ti$ gives rise to a set of generators $\Si$ for $\pi^\un_1(Z)^\ti$. If $\rho$ is a generator such that $\iota_*\rho \neq 0$, we denote $\iota_*\rho$ again by $\rho$. With this notational convention, $\Si$ is obtained from $\Si'$ simply by removing the generators $\tau_q$, $q\in Z\setminus Z'$, and $\iota^\sharp(f_w)$ (for $w$ any word in the generators $\Si$) is equal to $f_w$.

\subsection{}\label{cob5}
We return to the case $Z = \Spec \ZZ[1/6]$. In view of the remarks concerning functoriality (\ref{cob4}), the generators $\tau_2, \si$ of $\pi_1^\un(\ZZ[1/6])^\ti$ may be viewed as generators also of $\pi_1^\un(\ZZ[1/2])^\ti$. The $\QQ$-vector space $A(\ZZ[1/2])^\ti_3$ is spanned by the two subspaces
\[
E(\ZZ[1/2])_3 \quad
\mbox{and}
\quad
D(\ZZ[1/2])_3.
\]
Thus, as an element of $\pi_1^\un(\ZZ[1/2])^\ti$, $\si$ may be characterized as the unique element of $\Uu(\ZZ[1/2])^\ti_{-3}$ which pairs trivially with $D(\ZZ[1/2])_3$ and pairs to $1$ with $\zeu(3)$. In particular, it does not depend on any arbitrary choices. 

\begin{sprop}\label{cob6}
In the situation and the notation of segments \ref{cob1}-\ref{cob5}, we have
\begin{align*}
f_{\tau_q} 
	&= \logu(q) \quad (q = 2,3),
\\
f_{\si \tau_2} 
	&= - \frac{7}{8} \left(
	\frac{\logu(2)^4}{24} + \Liu_4(1/2)
	\right),
\\
f_{\si \tau_3} 
	&= \frac{3}{13} \left(
	6 \Liu_4(3) - \frac{1}{4} \Liu_4(9)
	\right).
\end{align*}
\end{sprop}

\begin{proof}
See section 4.3 of \cite{PolGonI}. The discussion of functoriality in segments \ref{cob4}, \ref{cob5} above, shows that the second equation, which, as interpreted in loc. cit., takes place in $A(\ZZ[1/2])^\times$, holds equally in $A(\ZZ[1/6])^\times$, with no conflict of notation. 
\end{proof}

\sUntitled{}{
Since 
\[
\logu(2) \zeu(3)=
f_\tau \sha f_\si = f_{\tau \si} + f_{\si \tau}
\]
and similarly for $f_{\up}$
we obtain
\begin{align*}
f_{\tau \si} 
	&= 
	\logu(2)\zeu(3) + 
 	\frac{7}{8} \left(
	\frac{\logu(2)^4}{24} + \Liu_4(1/2)
	\right) 
	\\
	&= (7/8) \Liu_4(1/2) + (7/192) \logu(2)^4 + \logu(2)\zeu(3)
\end{align*}
and
\begin{align*}
f_{\up \si} 
	&=
	\logu(3)\zeu(3)
	- \frac{3}{13} \left(
	6 \Liu_4(3) - \frac{1}{4} \Liu_4(9)
	\right)
	\\
	&= \logu(3)\zeu(3)
	- (18/13) \Liu_4(3) + (3/52) \Liu_4(9).
\end{align*}
}

\sProposition{22jy5}{
Let $Z \subset \Spec \ZZ$ be an open subscheme with complement 
\[
S = (\Spec \ZZ) \setminus Z.
\] 
As above, we let $\Uu(Z)^\times$ denote the completed universal enveloping algebra of $\pi_1^\un(Z)^\times$. For each $q \in S$, let $\tau_q \in \Uu(Z)^\times_{-1}$ be the unique element such that $\langle \logu(q'), \tau_q \rangle = 1$ if $q = q'$ and $0$ otherwise. For $n$ odd $\ge 3$, pick arbitrarily free generators $\si_n \in \Uu(Z)^\times_{-n}$ such that $\langle \zeu(n), \si_n \rangle = 1$. For a word $w$ in the free generators $\tau_q, \si_n$, let $f_w$ denote the element of $A(Z)^\times$ dual to $w$ with respect to the basis formed by such words. We denote the $q$-adic valuation on $\QQ$ associated to a prime $q$ by $v_q$. Then, independently of the choice of generators $\si_n$, 
\[
\logu (q) = f_{\tau_q},
\]
\[
\zeu(n) = f_{\si_n},
\]
and for any $a \in (\thrpl)(Z)$ and any $n \ge 1$,
\begin{align*}
\Liu_n(a) = 
\underset{q_1, \dots, q_{n-r} \in S} 
{\sum_{3 \le \, r \text{ odd} \, \le n}}
\langle \Liu_r(a), \si_r \rangle 
& v_{q_1}(a) \cdots v_{q_{n-r}}(a)
f_{\si_r \tau_{q_1} \cdots \tau_{q_{n-r}}} \\
& - \sum_{q_1, \dots, q_n \in S}
v_{q_1}(1-a) v_{q_2}(a) \cdots v_{q_{n}}(a)
f_{\tau_{q_1} \cdots \tau_{q_n}}.
\end{align*}
}

\begin{proof}
This is (a corrected version of) Remark 5.3 of \cite{PolGonII}.
\end{proof}

\subsection{}
\label{23jy1}
We allow ourselves to replace words in $\tau_2, \tau_3$ with words in $2,3$. Applying Proposition \ref{22jy5}, in half-weight 2 we obtain 
\begin{align*}
(\logu 2)^2 &= 2 f_{22} \\
(\logu 2)(\logu 3) &= f_{23} + f_{32} \\
(\logu 3)^2 &= 2 f_{33} \\
\Liu_2(-2) &= -f_{32}, 
\end{align*}
hence,
\begin{align*}
f_{22} &= \frac{1}{2} (\logu 2)^2 \\
f_{23} &= (\logu 2)(\logu3) + \Liu_2(-2) \\
f_{32} &= - \Liu_2(-2) \\
f_{33} &= \frac{1}{2} (\logu 3)^2
\end{align*}
and in half-weight 3 we obtain,
\begin{align*}
(\logu 2)^3 &= 6 f_{222} \\
(\logu2)^2(\logu3) &= 2(f_{223}+f_{232}+f_{322}) \\
(\logu2)(\logu3)^2 &= 2(f_{233}+f_{323}+f_{332}) \\
(\logu3)^3 &= 6f_{333} \\
(\logu2)\Liu_2(-2) &= -(f_{232}+2f_{322}) \\
(\logu3)\Liu_2(-2) &= -(2f_{332}+f_{323}) \\
\Liu_3(-2) &= -f_{322} \\
\Liu_3(3) &= -f_{233} \\
\zeu(3) &= f_\si,
\end{align*}
hence
\[
f_{223} = -\Liu_3(-2) + (\logu 2)\Liu_2(-2) +
\frac{1}{2} (\logu 2)^2 \logu 3
\]
\[
f_{233} = -\Liu_3(3).
\]

\subsection{}
In half weight 4, we expand only those polynomials in the polylogarithmic coordinates needed to convert the remaining shuffle coordinates $f_{2223}, f_{2233}, f_{2333}$. We find,
\begin{align*}
(\logu2)^3 (\logu 3) 
&= 6(f_{2223} + f_{2232} + f_{2322} + f_{3222})\\
(\logu2)^2 \Liu_2(-2)
&= -2(f_{2232}+2f_{2322}+3f_{3222}) \\
(\logu2)\Liu_3(-2)
&= -(f_{2322}+3f_{3222}).
\end{align*}
We are able to eliminate $f_{3222}$ with the help of
\begin{align*}
\Liu_4(-2) 
&= \langle \Liu_3(-2), \si \rangle v_2(-2)f_{\si 2} 
- v_3(3)v_2(-2)^3f_{3222} \\
&= -f_{3222},
\end{align*}
to obtain
\begin{align*}
f_{2223} = \Liu_4(-2) &- (\logu2)\Liu_3(-2) \\ 
&+\frac{1}{2} (\logu2)^2\Liu_2(-2) + \frac{1}{6} (\logu2)^3(\logu3).
\end{align*}

\subsection{}
\label{26jy1}
We have 
\begin{align*}
X(\ZZ[1/6]) = \{2,\frac{1}{2}, -1\} 
	& \cup \{3, \frac{1}{3}, \frac{2}{3}, \frac{3}{2}, 
	-\frac{1}{2}, -2\} \\
	& \cup \{4, \frac{1}{4}, \frac{4}{3}, 
	\frac{3}{4}, -\frac{1}{3}, -3\} \\
	&\cup \{-\frac{1}{8}, \frac{1}{9}, \frac{9}{8}, \frac{8}{9},
	9,-8 \}
\end{align*}
divided into $S_3$-orbits. We have for any $a \in X(Z)$ (or more generally, any $\Gm$-equivariant cocycle)
\[
\Delta' \Liu_n(a) = 
\sum_{i = 1}^{n-1} \Liu_{n-i}(a) \otimes
\frac{(\logu(a))^i }{i!}.
\]
Let $\Delta'_3$ denote the reduced coproduct 
\[
A_3(Z)^\times \to A_1(Z)^\times
\otimes A_2(Z)^\times \oplus A_2(Z)^\times \otimes A_1(Z)^\times
\]
and let $\Delta'_{1,2}$ denote its composite with the projection onto the factor $A_1(Z)^\times \otimes A_2(Z)^\times$. We recall that $E_i(Z) \subset A_i(Z)^\times$ denotes the space of extensions $\Ext^1_Z \big( \QQ(0), \QQ(i) \big)$, equal to the kernel of the reduced coproduct. We recall from Corollary 4.4 of \cite{PolGonI} that 
\[
\ker(\Delta'_{1,2}) = \ker (\Delta'_3) = E_3 = \QQ \zeu(3).
\]

\subsection{}
\label{26jy2}
We record shuffle decompositions of decomposables.
\begin{align*}
(\logu2)^2(\logu3)^2 &=
4f_{2233}+4f_{2323}+4f_{3223}
+4f_{2332}+4f_{3232}+4f_{3322}
\\
(\logu2)(\logu3)\Liu_2(-2) &=
-2f_{2332}-3f_{3232}-4f_{3322}-f_{2323}-2f_{3223}
\\
(\logu3)\Liu_3(-2) &= -2f_{3322}-f_{3232}-f_{3223}
\\
(\logu2)\Liu_3(3) &= -2f_{2233}-f_{2323}-f_{2332}
\end{align*}
We note the following relation between $f_{2233}$ and $f_{3322}$:
\[
(1/4)(\logu 2)^2 (\logu 3)^2
+ (\logu 3)\Liu_3(-2) + (\logu2) \Li_3(3)
= -f_{2233} - f_{3322}
\]

\subsection{}
\label{26jy4}
We record expansions of $\Delta'(L)$ for $L$ in the polylogarithmic basis $\{ (\logu2)^2, (\logu2)(\logu3), (\logu3)^2, \Liu_2(-2) \}$ of $A^\times_2$ in the basis for $A^\times_1$ in a table. We also include $\Delta'(\Liu_2(3))$ and $\Delta'(\Liu_2(\frac{2}{3}))$. We use the abbreviations $l$ for $\logu$ and $L$ for $\Liu$.
\[
\begin{matrix}
			& l(2)^2	& l(2)l(3)	& l(3)^2	& L_2(-2)	&& L_2(\frac{2}{3})	& \Liu_2(3)	\\
l(2)\otimes l(2)	& 2		& 0		& 0		& 0		&|& 0			& 0			\\
l(2)\otimes l(3)	& 0		& 1		& 0		& 0		&|& 0			& -1			\\
l(3)\otimes l(2)	& 0		& 1		& 0		& -1		&|& 1			& 0			\\
l(3)\otimes l(3)	& 0		& 0		& 2		& 0		&|& -1			& 0
\end{matrix}
\]
Using $\Delta': A_2 \xto{\sim} A_1\otimes A_1$, and the above table, we find 
\begin{align*}
\Liu_2(\frac{2}{3}) &= -\frac{1}{2} (\logu3)^2 - \Liu_2(-2) \\
\Liu_2(3) &= -\logu(2)\logu(3) - \Liu_2(-2).
\end{align*}

\subsection{}
\label{26jy3}
We write $\Liu_3(\frac{2}{3})$ as a polynomial in $\Ee \cup \Pp$. For each polylogarithmic basis element in $A_3(Z)^\times$, we expand $\Delta_{2,1}(L) \in A_2(Z)^\times \otimes A_1(Z)^\times$ in the basis induced by our polylogarithmic basis for $A_1(Z)^\times$ and $A_2(Z)^\times$. We record the result in a matrix along with the expansion of $\Delta_{2,1}(\Liu_3(\frac{2}{3}))$ in the rightmost column. 
\[
\tiny
\begin{matrix}
				& l(2)^3	& l(2)^2l(3)	&l(2)l(3)^2		& l(3)^3	&l(2)L_2(-2)	& l(3)L_2(-2)	& L_3(-2)	& L_3(3)	& L_3(\frac{2}{3})	\\
l(2)^2 \otimes l(2) 	&3		&0			&0			&0		&0			&0			&0		&0		&0				\\
l(2)^2 \otimes l(3) 	&0		&1			&0			&0		&0			&0			&0		&0		&0				\\
l(2)l(3) \otimes l(2) 	&0		&2			&0			&0		&-1			&0			&0		&0		&0				\\
l(2)l(3) \otimes l(3) 	&0		&0			&2			&0		&0			&0			&0		&-1		&0				\\
l(3)^2 \otimes l(2) 	&0		&0			&1			&0		&0			&-1			&0		&0		& -\frac{1}{2}		\\
l(3)^2 \otimes l(3) 	&0		&0			&0			&3		&0			&0			&0		&0		& \frac{1}{2}		\\
L_2(-2) \otimes l(2) 	&0		&0			&0			&0		&0			&0			&1		&0		& -1				\\
L_2(-2) \otimes l(3) 	&0		&0			&0			&0		&0			&1			&0		&-1		& 1
\end{matrix}
\]
From this and the exact sequence
\[
0 \to E_3(Z) \to A_3(Z)^\times \xto{\Delta_{2,1}} A_2(Z)^\times \otimes A_1(Z)^\times \to 0
\]
we find that 
\[
\Liu_3(\frac{2}{3}) \equiv 
-\frac 1 2 \logu(2)\logu(3)^2
+\frac 1 6 (\logu 3)^3
-\Liu_3(-2) - \Liu_3(3)
\mod \zeu(3).
\]

The Sage code

\begin{verbatim}
Q = Qp(13)

def l(z):
     return Q(z).log()

def Li(n,z):
     return Q(z).polylog(n)

def zeta(n):
     return 2^(n-1)*Li(n,-1)/(1-2^(n-1))

q = (Li(3,2/3) + (1/2)*l(2)*l(3)^2 \
- (1/6)*l(3)^3 + Li(3,-2) + Li(3,3))/zeta(3)

r = q.rational_reconstruction()
print(r)
\end{verbatim}
outputs the number $1$. Hence, at least up to the chosen precision, we have 
\[
\Liu_3(\frac{2}{3}) =
-\frac 1 2 \logu(2)\logu(3)^2
+\frac 1 6 (\logu 3)^3
-\Liu_3(-2) - \Liu_3(3)
+ \zeu(3).
\]

\subsection{}
We apply Proposition 5.8 to $\Li_{4}(2/3)$ and $\Li_{4}(4/3)$ :
\begin{multline} \label{eq:1} 
\Li_{4} (2/3) = \langle \Li_{3}(2/3),\sigma_{3} \rangle ( f_{\sigma_{3}\tau_{2}} - f_{\sigma_{3}\tau_{3}}) \\
+ f_{3222} - (f_{3322} + f_{3232} + f_{3223}) 
+ (f_{3233} + f_{3323} + f_{3332}) - f_{3333} 
\end{multline}
\begin{multline} \label{eq:2}
\Li_{4} (4/3) = \langle \Li_{3}(4/3),\sigma_{3} \rangle (2 f_{\sigma_{3}\tau_{2}} - f_{\sigma_{3}\tau_{3}}) \\
+ 8f_{3222} - 4(f_{3322} + f_{3232} + f_{3223}) 
+ 2(f_{3233} + f_{3323} + f_{3332}) - f_{3333} 
\end{multline}
We have 
\begin{equation} \label{eq:3}
f_{3233} + f_{3323} + f_{3332} = f_{2}f_{333} - f_{2333}.
\end{equation}
We have 
\begin{equation} \label{eq:4}
\log(3) \Li_{3}(-2) = - f_{3322} - (f_{3322} + f_{3232} + f_{3223}).\end{equation}

\subsection{}
By (\ref{eq:3}) and (\ref{eq:4}), we can regard (\ref{eq:1}) and (\ref{eq:2}) as a linear system of equations in $(f_{3322},f_{2333})$:
$$ \begin{array}{l}
\Li_{4} (2/3) - \langle \Li_{3}(2/3),\sigma_{3} \rangle ( f_{\sigma_{3}\tau_{2}} - f_{\sigma_{3}\tau_{3}})- f_{3222} -  \log(3) \Li_{3}(-2) - f_{2}f_{333} + f_{3333} 
\\= f_{3322} - f_{2333}
\\
\Li_{4} (4/3) - 
\langle \Li_{3}(4/3),\sigma_{3} \rangle 
(2 f_{\sigma_{3}\tau_{2}}  - f_{\sigma_{3}\tau_{3}}) 
- 8f_{3222} 
- 4 \log(3) \Li_{3}(-2) - 2f_{2}f_{333} 
+ f_{3333} 
\\
= 4f_{3322} - 2f_{2333}
\end{array} $$
If we denote 
\[
\Li_{4} (2/3) 
- \langle \Li_{3}(2/3),\sigma_{3} \rangle 
( f_{\sigma_{3}\tau_{2}} - f_{\sigma_{3}\tau_{3}})
- f_{3222} -  \log(3) \Li_{3}(-2) - f_{2}f_{333} + f_{3333}
\]
by $E$ and
\begin{align*}
\Li_{4} (4/3) 
- \langle \Li_{3}(4/3),\sigma_{3} \rangle 
(2 f_{\sigma_{3}\tau_{2}}  &- f_{\sigma_{3}\tau_{3}}) 
- 8f_{3222} \\&- 4 \log(3) \Li_{3}(-2) - 2f_{2}f_{333} + f_{3333}
\end{align*}
by $F$,
then we find that 
\[
\left\{\begin{array}{l} f_{3322} = - E + \frac{F}{2}
\\ f_{2333} = -\frac{E}{2} + \frac{F}{2}.
\end{array}
\right.
\]

The shuffle coordinates appearing in the above expressions for $E$ and $F$ have all been expanded in motivic polylogarithms above. It remains to compute the coefficient $\langle \Liu_3(4/3), \si_3 \rangle$.

\subsection{}
We have
$$
\Delta \Li_{3}(4/3) = \Li_{2}(4/3) \otimes \log(4/3) + \Li_{1}(4/3) \otimes \frac{1}{2} (\log(4/3))^{2}.
$$
By Proposition \ref{22jy5},
$$
\Li_{2}(4/3) = 2 f_{32} - f_{33}
$$
and
$$ f_{32} = - \Li_{2}(-2). $$
Thus, 
$$ \Li_{2}(4/3) = -2\Li_{2}(-2) - \frac{1}{2} (\log(3))^{2}.$$

Additionally, 
$$\Li_{1}(4/3) = - \log( 1 - 4/3) = - \log(-1/3) = - \log(-3) = - \log(3),$$
$$\log(4/3) = 2 \log(2) - \log(3),$$
$$\log(4/3)^{2} = 4 \log(2)^{2} - 4 \log(2)\log(3) + \log(3)^{2}.$$
So 
\begin{multline*} 
\Delta \Li_{3}(4/3) = 
(-2\Li_{2}(-2) - \frac{1}{2} (\log(3))^{2}) \otimes (2 \log(2) - \log(3)) 
\\ 
+ (-\log(3)) \otimes ( 2 \log(2)^{2} - 2 \log(2)\log(3) + \frac{1}{2} \log(3)^{2}).
\end{multline*}

\subsection{}
We deduce 
\begin{multline*} \Delta_{2,1} \Li_{3}(4/3) = 
(-2\Li_{2}(-2) \otimes (2 \log(2)
+ - \frac{1}{2} (\log(3))^{2}) \otimes (2 \log(2) 
\\ 
+ (-2\Li_{2}(-2) \otimes - \log(3))
+  - \frac{1}{2} (\log(3))^{2} \otimes  - \log(3)
\\
 = 
-4\Li_{2}(-2) \otimes \log(2)
 - \log(3)^{2} \otimes \log(2) 
\\ + 2 \Li_{2}(-2) \otimes  \log(3)
+ \frac{1}{2} (\log(3))^{2} \otimes \log(3).
\end{multline*}

We use the table of \S\ref{26jy3} and the fact that $\ker \Delta_{2,1}$ is generated by $\zeta(3)$. Let us denote by $A, B, C, D$ the four last lines of the table : 

$A = \log(3)^{2} \otimes \log(2)$

$B = \log(3)^{2} \otimes \log(3)$

$C= \Li_{2}(-2) \otimes \log(2)$

$D= \Li_{2}(-2) \otimes \log(3)$

We have found 
\begin{equation} \label{eq:kernel1} \Delta_{2,1} (\Li_{3}(4/3)) = - A + \frac{1}{2} B - 4C + 2D.
\end{equation}
According to the table, 

$\Delta_{2,1} (\log(2)\log(3)^{2} + 2\Li_{3}(3)) = A - 2D$

$\Delta_{2,1} (\log(3)^{3})  = 3B$

$\Delta_{2,1} (\Li_{3}(-2))  = C$

$\Delta_{2,1} (\log(3)\Li_{2}(-2))  = -A + D.$

Thus 

\begin{equation} \label{eq:kernel2} \left\{
\begin{array}{l} \Delta_{2,1} (-\log(2)\log(3)^{2} - 2\Li_{3}(3)+ 2\log(3)\Li_{2}(-2)) = A
	\\ 
\Delta_{2,1} (\frac{1}{3}\log(3)^{3})  = B
\\ 
\Delta_{2,1} (\Li_{3}(-2))  = C
\\ 
\Delta_{2,1} (-\log(2)\log(3)^{2} - 2\Li_{3}(3)- \log(3)\Li_{2}(-2))  = D
\end{array} \right.
\end{equation}

\subsection{}
Define 
\[
a = -\log(2)\log(3)^{2} - 2\Li_{3}(3) - 2\log(3)\Li_{2}(-2)
\]
\[
b= \log(2)\log(3)^{2} +2\Li_{3}(3) +\log(3)\Li_{2}(-2).
\]
Comparing (\ref{eq:kernel1}) and (\ref{eq:kernel2}) we deduce 
\begin{multline*} 
\Li_{3}(4/3) - (- (-\log(2)\log(3)^{2} - 2\Li_{3}(3) - 2\log(3)\Li_{2}(-2)) 
\\
+ \frac{1}{6} \log(3)^{3} - 4 \Li_{3}(-2) + 2(-\log(2)\log(3)^{2} - 2\Li_{3}(3)- \log(3)\Li_{2}(-2))) 
\\=
\Li_{3}(4/3) + 
\left(a - \frac{1}{6} \log(3)^{3} 
+ 4 \Li_{3}(-2) + 2b
\right) 
\in \ker \Delta_{2,1}
\end{multline*}

Thus the above element is a multiple of $\zeta(3)$. The coefficient is precisely $\langle \Li_{3}(4/3),\sigma_{3} \rangle$. Computation using a computer algebra system shows that the $p$-adic period of 
\[
\langle \Li_{3}(4/3),\sigma_{3} \rangle
=
\frac
{\Li_{3}(4/3) + 
\left(a - \frac{1}{6} \log(3)^{3} 
+ 4 \Li_{3}(-2) + 2b
\right) }
{\ze(3)}
\]
is equal to $-\frac 1 3$ for several primes (in particular, up to high $p$-adic precision).

\subsection{}
Assembling the previous computations, we have 
\begin{multline*} 
E=\Li_{4} (2/3) + \bigg( \frac{7}{8} \bigg( \frac{\log(2)^{4}}{24} + \Li_{4}(1/2) \bigg) + \frac{3}{13} \bigg( 6 \Li_{4}(3)- \frac{1}{4} \Li_{4}(9) \bigg)\bigg)
\\ 
+ \Li_{4}(-2) -  \log(3) \Li_{3}(-2) - \log(2)\frac{\log(3)^{3}}{3!} + \frac{\log(3)^{4}}{4!} 
\end{multline*}
\begin{multline*} F=\Li_{4} (4/3) -(1/3) 
\bigg(\frac{7}{4} \bigg( \frac{\log(2)^{4}}{24} + \Li_{4}(1/2) \bigg) 
+ \frac{3}{13} \bigg( 6 \Li_{4}(3)- \frac{1}{4} \Li_{4}(9) \bigg)\bigg) 
\\+ 8\Li_{4}(-2) - 4 \log(3) \Li_{3}(-2) - 2\log(2)\frac{\log(3)^{3}}{3!} + \frac{\log(3)^{4}}{4!} 
\end{multline*}

\subsection{}
\label{deceq}
Summarizing the results of our computations, we have the following equalities up to high precision. In longer equations we abbreviate $l = \logu, L = \Liu$.
\begin{align*} 
f_{\tau} = \logu(2) \\
f_{\up} = \logu(3) \\
f_{\tau \up} = (\logu 2)(\logu 3) + \Liu_2(-2) \\
f_\si = \zeu(3) \\
f_{\tau \tau \up} = -\Liu_3(-2) + (\logu 2)\Liu_2(-2) +
\frac{1}{2} (\logu 2)^2 \logu 3
 \\
f_{\tau \up \up} =  -\Liu_3(3) \\
f_{\tau \si} = (7/8) \Liu_4(1/2) + (7/192) \logu(2)^4 + \logu(2)\zeu(3) \\
f_{\up \si} =  \logu(3)\zeu(3)
	- (18/13) \Liu_4(3) + (3/52) \Liu_4(9)\\
f_{\tau \tau \tau \up} = 
\Liu_4(-2) - (\logu2)\Liu_3(-2) 
+(1/2) (\logu2)^2\Liu_2(-2) \\
+ (1/6) (\logu2)^3(\logu3)
\\
f_{\tau \tau \up \up} = (7/144)l(2)^4 - (1/4)l(2)^2l(3)^2 
+ (1/48)l(3)^4 \\
+ 2l(3)L_3(-2) + l2L_3(3) + L_4(2/3) 
\\
+ (7/6)L_4(1/2) + (3/2)L_4(3) - (1/16)L_4(9) \\
- 3L_4(-2) 
- (1/2)L_4(4/3)
\\
f_{\tau \up \up \up} =
-(35/1152)l(2)^4 - (1/12)l(2)l(3)^3 
\\
- (3/2)l(3)L_3(-2) 
- (1/2)L_4(2/3) \\- (35/48) L_4(1/2) 
- (12/13) L_4(3) \\+ (1/26) L_4(9) 
+ (7/2) L_4(-2) + (1/2) L_4(4/3)
\end{align*}

\section{The polylogarithmic quotient}
\label{polquosection}

Throughout this section, Tannakian fundamental groups are endowed with the usual \emph{functorial} product. 

\subsection{}\label{27n1}
The literature on motivic tangential base-points for unirational varieties of dimension greater than 1 is not fully fleshed out. The theory is nevertheless regarded as known, as it amounts to a fairly straightforward generalization of the 1-dimensional case, complemented by techniques for bootstrapping to higher dimensions in \S4 of Deligne-Goncharov \cite{DelGon}. $p$-Adic aspects are discussed in \"Unver \cite{UnverRelations}.
Here we provide an outline of the construction and verify that our integrality conditions on tangential base-points ensure that the associated fundamental groups are unramified. We begin with the $l$-adic realization. 

\subsection{}\label{27n2}
Let $Z$ be an open subscheme of $\Spec \ZZ$, $\overline X \to Z$ a smooth proper morphism whose generic fiber is unirational, $D \subset \overline X$ a relative simple normal crossings divisor whose irreducible components are smooth and absolutely irreducible, and let $X$ denote its complement in $\overline X$. By a $Z$-integral base-point we mean either a section of $X \to Z$ or a nonvanishing $Z$-family of tangent vectors along a stratum of $\overline X$ which are not tangent to any boundary divisor. 

\subsection{}\label{20606a}
Let us be more explicit about our assumptions on a $Z$-integral tangent vector $v$. We are provided with a $Z$-point
\[
x: Z \to \overline X
\]
of the compactification.  Let
\[
T_x = \Spec \Sym^* \Tt_x^\lor
\]
denote the normal bundle to $x$ in $\overline X$. We use $\Oo_Z$ both for the structure sheaf and for the coordinate ring of $Z$. We may equivalently think of $\Tt_x$ as a quasi-coherent sheaf on $Z$ or as a module over the coordinate ring $\Oo_Z$, and we do not distinguish between these notationally.  In this notation,
\[
v \in \Hom_Z(Z, T_x) = \Tt_x
\]
is a section which is \textit{nowhere tangent to the boundary divisors}.

This last phrase may be interpreted in several equivalent ways; what we need is the following. If $R$ is a ring, we denote by $R((t))$ the ring $R\llbracket t \rrbracket[t\inv]$ of Laurant series with coefficients in $R$. Let $\hat D_Z = \Spec \Oo_Z \llbracket t \rrbracket$, let
\[
\hat D_Z^\circ = \Spec \Oo_Z (( t)),
\]
and let $T_0$ denote the normal bundle to $\hat D$ along the zero section $\{t = 0\}$. Then there's a map
\[
h:\hat D \to \bar X
\]
which maps the zero section $\{t = 0\}$ to $x$, such that $\hat D^\circ_Z$ maps to $X$ and such that the induced map of normal bundles
\[
T_0  \to T_x
\quad
\mbox{maps}
\quad 
1 \mapsto v.
\]

Fix a prime $l$ of $\ZZ$. Let $\iota$ denote the natural map
\[
\hat D^\circ_Z \to \AA^1_Z \setminus \{0\}.
\]
Pullback along $\iota$ induces an equivalence of categories of finite \'etale coverings. Consequently, $\iota_*$ induces a monoidal equivalence of categories of lisse $\QQ_l$-sheaves. 

We let $\opnm{Lisse}_{\QQ_l}(\dots)$ denote the category of lisse sheaves. Let $g$ denote the map
\[
\Spec \overline \QQ \to Z
\]
induced by the choice of an algebraic closure $\overline \QQ$ of $\QQ$. The composite (diagonal solid arrow below)
\[
\xymatrix{
\Lisse_{\QQ_l}(X)
\ar[r]^-{h^*_v}
\ar@{.>}[rrd]^-{v^*}
\ar[rrdd]_-{\om_v}
&
\Lisse_{\QQ_l}(\hat D^\circ_Z)
\ar[r]^-{\iota_*}
&
\Lisse_{\QQ_l}(\AA^1_Z \setminus\{0\})
\ar[d]^-{1^*}
\\
&&
\Lisse_{\QQ_l}(Z)
\ar[d]^-{g^*}
\\
&&
\Vect \QQ_l
}
\]
defines a ``tangential'' fiber functor on the category of Lisse $\QQ_l$-sheaves associated to the $Z$-integral tangent vector $v$. We note the intermediate composite, denoted $v^*$, for future use. We also note that the same construction defines tangential fiber functors over various base-extensions of $X$ ($X_{\Qp}$, $X_{\overline \QQ}$, ...) and we continue to use the same notation $\om_v$. A similar construction at the level of Galois categories of finite \'etale coverings provides us with a notion of tangential fiber functors for profinite \'etale fundamental groups, and there's an obvious compatibility between the two constructions. 

We claim that 
\[
\pi_1^l(X_{\bar \QQ}, v) = 
\Aut^\otimes (\om_v)
\]
is unramified at $p$. By this we mean the following.

\begin{sprop}
\label{20610a}
Fix arbitrarily an embedding $\overline \QQ \subset \overline \QQ_p$. There's an associated decomposition group
\[
G_\QQ \supset G_p = G_{\QQ_p}. 
\]
Then the induced action of $G_p$ on $\pi_1^l(X_{\bar \QQ}, v)$ factors through
\[
G_p \surj G_{\FF_p} = \hat \ZZ.
\]
\end{sprop}

\begin{proof}
Our construction of the $l$-adic tangential fiber functor $\om_v$ applies equally over $Z_p = \Spec \Zp$ and the verification of the above statement may take place over $Z_p$. For this purpose we temporarily replace $Z$ by $Z_p$. We let $f$ denote the structure morphism
\[
X \to Z_p.
\]
We define a lisse $\QQ_l$-sheaf $\Ff$ on $X$ to be \emph{relatively unipotent} if $\Ff$ admits a filtration by lisse subsheaves such that
\[
\gr \Ff \simeq f^* \Gg
\]
for some lisse sheaf $\Gg$ on $Z_p$. We define \emph{relatively lisse sheaves on $X_\Qp$} similarly. We decorate ``un'' to indicate full subcategories of unipotent objects and ``run'' to indicate full subcategories of relatively unipotent objects. 
We have a diagram of $\QQ_l$-Tannakian categories
\[
\xymatrix{
&
\Lisse^\m{un}_{\QQ_l} X_{\overline \QQ_p} 
\ar@/^3ex/[d]^{v^*}
&
\Lisse^\m{run}_{\QQ_l} X_\Qp
\ar[l] \ar@/^3ex/[d]^{v^*}
&
\Lisse^\m{run}_{\QQ_l} X
\ar@/^3ex/[d]^{v^*}
\ar[l]
\\
\Vect \QQ_l
&
\Lisse_{\QQ_l} \Spec \overline \QQ_p
\ar[l]_-\sim
\ar[u]
&
\Lisse_{\QQ_l} \Spec \Qp
\ar[l] \ar[u]
&
\Lisse_{\QQ_l} Z_p,
\ar[u]
\ar[l]
}
\]
which is filled in by canonical natural $\otimes$-isomorphisms. The functor to $\Vect \QQ_l$ endows each of the Tannakian categories appearing in the diagram with a fiber functor which we use as base-point for Tannakian fundamental groups and gives rise to a morphism of split short exact sequences of Tannakian fundamental groups
\begin{tiny}
\[
\xymatrix{
1
\ar[r]
&
\pi_1 \big(
\Lisse^\m{un}_{\QQ_l} X_{\overline \QQ_p}
\big)
\ar[r]
&
\pi_1 \big(
\Lisse^\m{run}_{\QQ_l} X
\big)
\ar[r]
&
\ar@/^3ex/[l]
\pi_1 \big(
\Lisse_{\QQ_l} Z_p
\big)
\ar[r]
&
1
\\
1
\ar[r]
&
\pi_1 \big(
\Lisse^\m{un}_{\QQ_l} X_{\overline \QQ_p}
\big)
\ar[r] \ar@{=}[u]
&
\pi_1 \big(
\Lisse^\m{run}_{\QQ_l} X_{\Qp}
\big)
\ar[r] \ar[u]
&
\ar@/^3ex/[l]
\pi_1 \big(
\Lisse_{\QQ_l} \Spec \Qp
\big)
\ar[r] \ar[u]
&
1.
}
\]
\end{tiny}
On the other hand, the natural transformation from profinite \'etale fundamental groups to Tannakian $l$-adic fundamental groups provides a commuting square
\[
\xymatrix{
\pi_1(\Lisse_{\QQ_l} Z_p)(\QQ_l)
&
**[r]\pi_1^\m{\acute{e}t}(Z_p, a) \simeq 
\opnm{Gal}(\overline \FF_p /\FF_p)
\ar[l]
\\
\pi_1(\Lisse_{\QQ_l} \Spec \Qp)(\QQ_l)
\ar[u]
&
**[r]\pi_1^\m{\acute{e}t}(\Spec \Qp, a) \simeq 
\opnm{Gal}(\overline \QQ_p /\QQ_p).
\ar[l] \ar[u]
}
\]
where $a$ denotes the base-point associated with our choice of algebraic closure. This shows that the action of $\opnm{Gal}(\overline \QQ_p /\QQ_p)$ on 
$
\pi_1 \big(
\Lisse^\m{un}_{\QQ_l} X_{\overline \QQ_p}
\big)
$
factors through $\opnm{Gal}(\overline \FF_p /\FF_p)$ as claimed.
\end{proof}

\subsection{}{}
Similar constructions to the one outlined above provide Betti and de Rham versions of the unipotent fundamental group at a tangential base-point; see, for instance, \S\ref{cca2} below. A mixed Hodge structure on the Betti unipotent fundamental group at a tangential base-point is constructed in works of Hain and collaborators. For instance, Definition 4.21(ii) of Hain-Zucker \cite{HainZucker} provides a structure of pro-\textit{variation of mixed Hodge structures} on the bundle whose fiber at a point $x$ is the prounipotent completion of the fundamental group at $x$; restricting to an appropriate analytic disk and taking a limit mixed Hodge structure, one obtains a mixed Hodge structure at a tangential base-point.

Together with natural comparison isomorphisms, this provides a prounipotent group object $\pi_1^{H+l}(X_\QQ, v)$ in the $\QQ$-Tannakian category $\Rr^{H+l}$ of systems of realizations of type $H+l$ considered in Deligne-Goncharov \cite{DelGon}. The methods of \S4 of loc. cit. provide a unipotent group object $\pi_1^\un(X_\QQ, v)$ in the Tannakian category of mixed Tate motives over $\Spec \QQ$ which realizes to $\pi_1^{H+l}(X_\QQ, v)$. This is an elaboration on Remark 4.14 of loc. cit.

By proposition 1.8 of loc. cit., to show that $\pi_1^\un(X_\QQ, v)$ belongs to the full subcategory of mixed Tate motives over $Z$, it's enough to check that at each $p \in Z$, an $l$-adic realization ($l \neq p$) is unramified, as was done in Proposition \ref{20610a}. This amounts to an elaboration on Remark 4.18 of loc. cit. We write $\pi_1^\un(X,v)$ when we regard the unipotent fundamental group as a prounipotent group object of $\MT(Z)$.

\subsection{}
For a pair of $Z$-integral base-points $a,b$, similar constructions to the ones outlined above provide a $\pi_1^\un(X,a)$-$\pi_1^\un(X,b)$-bitorsor ${_b P _a}$ of \emph{motivic paths from $a$ to $b$}. There are \emph{path composition morphisms}
\[
{_c P_b} \times {_b P_a} \to 
{_c P_a}; 
\]
if $\Bb$ is a set of $Z$-integral base-points, then the collection
\[
\{{_b P_a}\}_{a,b \in \Bb}
\]
has the structure of a \emph{groupoid in $\MT(Z)$} in an obvious sense.

\subsection{}\label{20606b}
If $E$ is an object of $\MT(Z)$, we denote by $\VV^\lor E$ the associated vector group object --- its image under any fiber functor $\om$ is given by
\[
\om(\VV^\lor E) = \Spec \Sym^* \om(E)^\lor.
\]
If $y$ is a tangential base-point whose support is contained in a boundary-divisor, then there's an associated \emph{local monodromy morphism} 
\[
\VV^\lor \QQ(1)^R \to \pi_1^\un(X,y)^R
\]
in any realization $R$. When $X$ is a curve, a motivic version of the local monodromy morphism figures into the very construction of Deligne-Goncharov \cite[\S4.3-4.11]{DelGon}.\footnote{Moreover, some of the issues dealt with in loc. cit. are simplified when working over $Z = \Spec \ZZ$, where there are no nontrivial Kummer motives.} More generally, when there exists a map of pointed-varieties 
\[
(C,c) \to (X,y)
\]
over $Z$, with $C$ a mixed Tate curve and $c$ a $Z$-integral tangent vector along a puncture, one obtains a motivic local monodromy morphism by composing
\[
\VV^\lor \QQ(1) \to \pi_1^\un(C,c) \to \pi_1^\un(X,y).
\]
There's an independence of choice of $(C,c)$, which may be checked in realizations. In the cases considered below, this construction can be made particularly concrete by always taking $C$ to be $\PP^1$ minus a finite collection of disjoint $Z$-sections. 

Suppose $Y \to Z$ satisfies the same conditions as $X\to Z$, $f:X \to Y$ is a morphism over $Z$, and $x_1,x_2$ are $Z$-integral base-points of $X$ whose images $y_1, y_2$ are $Z$-integral base-points of $Y$. Then there's an induced morphism of affine groupoids in mixed Tate motives from the groupoid formed by the base-points $x_1,x_2$ to the groupoid formed by the base-points $y_1,y_2$. 

The local monodromy morphisms and the functoriality are compatible in the following sense: if $f:X \to Y$ sends the tangential base-point $x$ to the tangential base-point $y$ then the local monodromy morphism associated to $y$ is equal to the composite 
\[
\VV^\lor \QQ(1) \to \pi_1^\un(X,x) \to \pi_1^\un(Y,y).
\]
Indeed, this may be checked in any realization, where it becomes evident. 

\subsection{}\label{cca1}
Let $Y$ denote the complement of the 5 divisors
\[
D_1 = \{z_1 = 0\}, \quad
D_{11}=\{z_1 = 1\}, \quad
D_{2} = \{z_2 = 0\}, \quad
D_{22}=\{z_2 = 1\},
\]
and
\[
D_{12} = \{z_1z_2 = 1\}
\]
inside $\AA^2_\ZZ$. Let $j$ denote the natural inclusion 
\[
Y \inj \AA^2 \setminus (D_1 \cup D_2) = \Gm \times \Gm
\]
in the complement of the cross-hairs $+$. Let $(1,1)_0$ denote the tangent vector $(1,1)$ at $0$. 

If $b$ is any base-point, we let $K(b)$ denote the kernel of the induced map of unipotent fundamental groups
\[
\pi_1^\un \big( Y, b \big) \to 
\pi_1^\un \big( \AA^2 \setminus (D_1 \cup D_2), b \big).
\]
Let
\[
\pi^\PL(Y, b):= \pi_1^\un \big( Y, b \big) / [K(b),K(b)]
\]
and let 
\[
\nN^\PL(Y, b) := \Lie \pi^\PL(Y,b).
\]
When $b = (1,1)_0$ we simply write $\pi^\PL(Y)$, $\nN^\PL(Y)$. Our goal for this section is to establish the following

The pro-object
\[
\big( \QQ(1) \big)^2 \ltimes 
\left( \prod_{i=1}^\infty \QQ(i) \right)^3
\]
in mixed Tate motives has a natural structure of Lie algebra: the factors on both sides of the semidirect product are abelian, and the bracket between factors on the left and factors on the right is induced by the canonical isomorphisms
\[
\tag{*}
\QQ(1) \otimes \QQ(i)
\xto{\simeq}
\QQ(i+1)
\]
as follows: the first factor on the left acts on the first and third copies of the product $\prod_{i=1}^\infty \QQ(i)$ via (*) and on the second copy by zero, while the second factor on the left acts on the second and third copies of the product via (*) and on the first copy by zero.

\begin{sprop} 
\label{polquoprop}
In the situation and the notation above, there is an isomorphism of Lie algebra objects in the category of mixed Tate motives over $\ZZ$ 
\[
\tag{*}
\nN^\PL(Y) = 
\big( \QQ(1) \big)^2 \ltimes 
\left( \prod_{i=1}^\infty \QQ(i) \right)^3.
\]
\end{sprop}

In view of Proposition \ref{polquoprop}, $\nN^\PL(Y)$ may be thought of as the abstract polylogarithmic Lie algebra of \S\ref{ab2} equipped with a (quite trivial) motivic Galois action. From now on, we allow ourselves to write $\nN^\PL$, $\pi^\PL$, etc. in place of $\nN^\PL(Y)$, $\pi^\PL(Y)$, etc. 

\subsection{}\label{ccb1}
We begin by recalling well-known facts about $M_{0,5}$.
There's an isomorphism $M_{0,5} = Y$, and hence an open immersion 
\[
\ka: M_{0,5} \inj \PP^1 \times \PP^1
\]
with complement the 7 divisors 
\[
D_1 = \{z_1 = 0\}, \quad
D_{11}=\{z_1 = 1\}, \quad
D_{2} = \{z_2 = 0\}, \quad
D_{22}=\{z_2 = 1\},
\]
\[
D_{12} = \{z_1z_2 = 1\}, \quad
\{z_1 = \infty\}, \quad
\{z_2 = \infty\}.
\]
We let 
\[
M_{0,5} \inj \overline M_{0,5}
\]
denote the Deligne-Mumford compactification. The map $\ka$ extends to a map $\overline \ka$ 
\[
\xymatrix{
& \overline M_{0,5} \ar[d]^{\overline \ka}  \\
M_{0,5} \ar[ur] \ar[r]_-\ka & \PP^1 \times \PP^1
}
\]
which identifies $ \overline M_{0,5}$ with the blowup of $\PP^1 \times \PP^1$ at the three points $(1,1)$, $(0,\infty)$, $(\infty,0)$. In particular, $\overline M_{0,5}$ has 3 exceptional divisors in addition to the 7 boundary divisors listed above. These are all isomorphic to $\PP^1$ over $\Spec \ZZ$ and have strict normal crossings so that the formal neighborhood of each intersection is isomorphic to $\Spec \ZZ\llbracket t,u \rrbracket$ with divisors given by $t=0$ and $u=0$. In particular, there are 4 $\ZZ$-integral tangential base-points associated to each point of intersection ($(\partial_t,\partial_u) = (\pm 1, \pm 1)$).

\subsection{}\label{cca3}
 The 1-forms 
\[
\xi_1 = \frac{dz_1}{z_1}, \;
\xi_{11} = \frac{dz_1}{1-z_1}, \;
\xi_2 = \frac{dz_2}{z_2}, \;
\xi_{22} = \frac{dz_2}{1-z_2}, \;
\xi_{12} = \frac{d(z_1z_2)}{1-z_1z_2}
\]
form a basis of $H^1_\dR(Y_\QQ)$. The construction of Deligne \cite[\S12]{Deligne89} provides a canonical splitting of the natural surjection 
\[
\pi_1^\dR (Y, y) \surj H_1^\dR(Y_\QQ)
\]
for any base-point $y$. Together, the basis and the splitting give rise to a canonical surjection
\[
\tag{*}
\pi(\eo, \et, \eoo, \ett, \eot) \surj \pi_1^\dR(Y, y)
\]
from the free prounipotent group on the set of generators $\Gamma$. 

\subsection{}\label{cca2}
In the de Rham setting, as in the $l$-adic setting, tangential fiber functors and local monodromy morphisms may be obtained directly from the 1-dimensional construction. In \cite[\S15.28-36]{Deligne89} Deligne constructs a functor from the category of vector bundles with integrable connection on $\hat D_\QQ^*:=\Spec \QQ (( t ))$ with regular singularity at $t=0$ to the category of vector bundles with integrable connection on $\GG_{m,\QQ}$. Let us denote this functor by $\delta$. 

Let $\VICinfty$ denote the category of unipotent vector bundles with integrable connection. Recall that a unipotent vector bundle with integrable connection on the complement of a simple normal crossings divisor inside a smooth scheme automatically has regular singularities along the divisor. The same holds for the divisor $t=0$ inside the formally smooth $\QQ$-scheme $\hat D_\QQ = \Spec \QQ \llbracket t \rrbracket$.

We let
\[
\om_{1}: \VICinfty(\GG_{m,\QQ}) \to \Vect(\QQ)
\]
denote the fiber functor 
\[
(E, \nabla) \mapsto E(1)
\]
associated to the point $1 \in \Gm$. If $y$ is a tangential base point of $M_{0,5}$ supported at $\overline y \in \overline M_{0,5}$, we let 
\[
h_y: \hat D_\QQ
\to \overline M_{0,5}
\]
be a map sending the closed point to $\overline y$ and whose derivative sends $1$ to $y$. We let $h_y^o$ denote the induced map
\[
\hat D_\QQ^* \to M_{0,5} = Y.
\]
In terms of the maps and functors defined above, we define 
\[
\om_{y} : \VICinfty(Y) \to \Vect(\QQ)
\]
to be the composite
\[
\VICinfty(Y) \xto{(h_y^o)^*} 
\VICinfty(\hat D^*_\QQ) \xto{\de}
\VICinfty(\GG_{m,\QQ}) \xto{\om_1}
\Vect(\QQ).
\]
Then $\pi_1^\dR(Y,y) = \Aut^\otimes(\om_y)$ is the de Rham realization of $\pi_1^\un(Y,y)$. 

If $y$ is a tangential base-point supported at a point contained in the divisor associated to the generator $e \in \{e_1, e_2, e_{11}, e_{22}, e_{12}\}$ then the associated local monodromy morphism in mixed Tate motives realizes to the composite 
\[
\Ga = \pi(e) \subset \pi(\eo, \et, \eoo, \ett, \eot) \surj \pi_1^\dR(Y, y)
\]
In terms of the presentation \ref{cca3}(*), $\pi^\PL$ is the prounipotent group associated to the ``abstract polylogarithmic Lie algebra'' considered in \S \ref{ap0} above.

\subsection{}\label{27n5}

By construction, the kernel $K$ of the projection 
\[
\phi:
\pi^\PL \surj \VV^\lor \QQ(1)^2
\]
is commutative. The local monodromy morphisms associated to the divisors $D_1$, $D_2$ induce a splitting $\mu = \mu_1 \oplus \mu_2$ of $\phi$. Our next goal is to construct a morphism
\[
\mu': \VV^\lor \QQ(1)^3 \to K
\]
associated to monodromy about the divisors $D_{11}$, $D_{22}$, $D_{12}$.

\subsection{} \label{ccc2}
We will construct a map 
\[
\mu_{12}: \VV^\lor \QQ(1) \to K \subset \pi^\PL(Y)
\]
corresponding to monodromy about the divisor $D_{12}$. The same construction, mutatis mutandis, provides similar maps $\mu_{11}$, $\mu_{22}$ corresponding to monodromy about the divisors $D_{11}$, $D_{22}$, respectively. The map $\mu'$ is then the direct sum
\[
\mu' = \mu_{11} \oplus \mu_{22} \oplus \mu_{12}.
\]
Let $y$ be a $\ZZ$-integral tangential base-point supported along the intersection of $D_{12}$ with the exceptional divisor $E$ over the point $(1,1) \in Y$ (recall from segment \ref{ccb1} above that there are precisely 4 such). Since $y$ is nowhere tangent to $E$, its image $w$ in the relative tangent bundle to $\AA^2 \setminus (D_1 \cup D_2)$ along the $\ZZ$-point $(1,1)$ is again a $\ZZ$-integral tangential base-point. The composite of the associated local monodromy map
\[
\tag{*}
\VV^\lor\QQ(1) \to \pi^\PL(Y,y)
\]
with the map
\[
\pi^\PL(Y,y) \to
 \pi_1^\un(\AA^2 \setminus(D_1 \cup D_2), y) 
 = \VV^\lor \QQ(1)^2
\]
is zero, so the local monodromy map factors through a map
\[
\tag{**}
\VV^\lor \QQ(1) \to K(y).
\]
The action of $\pi^\PL(Y)$ on $K$ factors through $\pi_1^\un(\AA^2 \setminus (D_1 \cup D_2), (1,1))$. Thus, $K(y)$ is equal to $K$ twisted by the torsor
\[
\tag{$\tau$}
\pi_1^\un(\Gm\times \Gm, (1,1)_0, y).
\]
Since there are no nontrivial Kummer motives over $\Spec \ZZ$, this torsor is trivial. Hence there's a canonical isomorphism of commutative unipotent group objects
\[
K(y)=K.
\]
Composing with (**) we obtain the map $\mu_{12}$. 

\subsection{}\label{ccd1}
The Lie bracket is a morphism of pro-\textit{mixed Tate motives} 
\[
[\cdot,\cdot]:\nN^\PL(Y) \otimes \nN^\PL(Y) \to \nN^\PL(Y).
\]
Let $\nu_{1}$, $\nu_{11}$, etc. be the maps of Lie algebras associated to the morphisms $\mu_?$ constructed above. For $n \ge 1$ we let 
\[
\nu_{11,n} := (\ad \nu_{1})^{n-1}(\nu_{11}), 
\]
that is, $\nu_{11,n}$ is the map
\[
\QQ(n) 
\xto{\nu_1 \otimes \cdots \otimes \nu_1 \otimes \nu_{11}}
\nN^\PL(Y)^{\otimes n}
\xto{[ \cdot, \cdots [\cdot ,[\cdot,\cdot]] \cdots ]}
\nN^\PL(Y).
\]
Similarly, we let 
\[
\nu_{22,n} := (\ad \nu_{2})^{n-1}(\nu_{22}),
\quad
\nu_{12,n} := (\ad \nu_{1})^{n-1}(\nu_{12}).  
\]
We could also define 
\[
\nu'_{12,n} := (\ad \nu_{2})^{n-1}(\nu_{12}).
\]
That $\nu_{12,n} = \nu'_{12,n}$ may be checked after passage to de Rham realization, where it's \ref{echo1} above. 
Together, the maps $\nu_?$ define a morphism of pro-\textit{mixed Tate motives}
\[
\tag{*}
\nN^\PL(Y) \xfrom{\nu}
\big( \QQ(1) \big)^2 \ltimes 
\left( \prod_{i=1}^\infty \QQ(i) \right)^3.
\]
We may check that $\nu$ is an isomorphism of Lie algebra objects after passage to de Rham realization where it follows from \ref{cca3}(*) and \ref{liebasis}, in view of the known computation of the de Rham fundamental group of $M_{0,5}$. 

This last computation may be extracted from the literature for instance as follows. \"Unver \cite[\S5]{UnverIhara} constructs generators $E_{i,j}$ ($0 \le i, j \le 4$) for the Lie algebra of the de Rham fundamental group, and proves that the latter is free pronilpotent on these generators modulo the relations 
\begin{align*}
 E_{ii} &= 0, \\
 E_{ji} &= - E_{ij}, \\
 \Si_{i} E_{ij} &= 0, \\
 [E_{ij}, E_{kl}] &=0 
 \mbox{ whenever } 
 \{i,j\} \cap \{k,l\} = \emptyset.
\end{align*}
(We have capitalized \"Unver's ``$E_{ij}$'' in order to avoid a conflict with our notation.) The generators are determined by their action on the universal prounipotent connection (the ``KZ'' connection). A presentation of the latter which makes the action evident is given by Oi-Ueno in \S2.1 of \cite{OiUeno} (where the generators $E_{ij}$ of \"Unver are denoted by $\Om_{ij}$). In terms of these generators, ours are given by
\begin{align*}
e_1 &= E_{12}+E_{13}+E_{14},
\\
e_2 &= E_{23},
\\
e_{11} &= -E_{14},
\\
e_{22} &= -E_{12},
\\
e_{12} &= -E_{24},
\end{align*}
as may be seen, for instance, by computing their action on the KZ-connection. The implied relations \ref{1.1}(R) are listed in \cite[\S4.1]{OiUeno}. This completes the proof of Proposition \ref{polquoprop}.

\begin{sremark}
Recall that $Y$ denotes the moduli space $M_{0,5}$ in its guise as
\[
\Spec \ZZ[z_1,z_2, z_1\inv, (1-z_1)\inv, z_2\inv, (1-z_2)\inv, (1-z_1z_2)\inv].
\]
Let $X = M_{0,4} = \Spec \ZZ[x, x\inv, (1-x)\inv]$. The map
\[
\iota:Y \to X^3
\]
\[
(z_1,z_2) \mapsto (z_1, z_2, z_1 z_2)
\]
is a closed immersion with image the closed subscheme defined by the equation
\[
x_3 = x_1 x_2.
\]
The de Rham first cohomology vector space $H^1_\dR(X^3_\QQ) = H^1_\dR(X)^{\oplus 3}$ has basis the six 1-forms 
\[
\tag{*}
\frac{dx_i}{x_i}, \quad \frac{dx_i}{1-x_i} \quad (i = 1,2,3). 
\]
Their pullbacks along $\iota$ span $H^1_\dR(Y_\QQ)$ and are linearly independent modulo the one relation
\[
\tag{**}
\iota^* \frac{dx_3}{x_3}
= \iota^* \frac{dx_1}{x_1} + \iota^*\frac{dx_2}{x_2}.
\]
If we label the six generators of the de Rham unipotent fundamental group $\pi_1^\dR(X^3)$ of $X^3$ at the tangential base-point $(\vec{1_0},\vec{1_0},\vec{1_0})$ associated to the 1-forms (*) as follows:
\[
d_0^1, d_1^1, d_0^2, d_1^2, d_0^3, d_1^3,
\]
then the map of de Rham unipotent fundamental groups
\[
\pi_1^\dR(Y) \to \pi_1^\dR(X^3)
\]
sends
\begin{align*}
e_1 &\mapsto d^1_0 + d^3_0 \\
e_{11} &\mapsto d^1_1 \\
e_2 &\mapsto d^2_0 + d^3_0 \\
e_{22} &\mapsto d^2_1 \\
e_{12} &\mapsto d^3_1. \\
\end{align*}
In terms of the associated map of Hopf algebras (with dual elements in the Hopf algebra denoted by $f_?$ as usual) equation (*) reads 
\[
\iota^* f_{d^3_0}  = 
\iota^*f_{d^1_0} +\iota^*f_{d^2_0}.
\]
This gives geometric meaning to the equation ``$f_{e_3} = f_{e_1} +f_{e_2}$''.

Let $\pi_1^\PL(X^3) = \pi_1^\PL(X)^3$ denote the quotient of $\pi_1^\dR(X^3)$ associated to the polylogarithmic quotient of $\pi_1^\dR(X)$ (or a quotient thereof by some step of the descending central series). Let $Z = \Spec \ZZ[1/6]$ as usual, and let $K(Z)$ denote the fraction field of the prounipotent mixed Tate Galois group $\pi_1^\un(Z)$. Let $\Aa(X^3) = \Aa(X)^{\otimes 3}$ denote the coordinate ring of $\pi^\PL(X^3)_{K(Z)}$. Let $\Ss(X)$ denote the coordinate ring of the base-change to $K(Z)$ of the Selmer scheme
\[
{\bf H}^1 \big( \pi_1^{\MT}(Z), \pi^\PL(X) \big)
= {\bf Z}^1\big(
\pi_1^\un(Z), \pi^\PL(X) 
\big)^{\Gm}
\]
and similarly for $X^3$; we have
\[
\Ss(X^3) = \Ss(X)^{\otimes 3}.
\]
Then the universal cocycle evaluation maps of $Y$ and of $X^3$, together with the maps $\iota_\Aa$, $\iota_\Ss$ induced by the embedding $\iota$, form a commuting square of $K(Z)$-algebras
\[
\xymatrix{
\Ss(Y) & \ar[l]_-{\theta_Y} \Aa(Y)\\
\Ss(X^3) \ar[u]^-{\iota_\Ss} & \ar[l]^-{\theta_{X^3}} \Aa(X^3)  \ar[u]_-{\iota_\Aa}.
}
\]

Let $d_0$, $d_1$ denote the standard generators of the de Rham unipotent fundamental group of $X$. In Proposition \S\ref{6s2} we essentially constructed a certain polynomial $p(y) \in \Aa(X)[y]$ such that (after translating along the universal cocycle evaluation map $\theta_X$)
\[
\tag{R}
p(\Phi^\tau_{d_0}) = 0
\]
where $\Phi^\tau_{d_0}$ denotes the function on cocycles 
\[
\Phi^\tau_{d_0}(c) = \langle \tau, c^\sharp d_0 \rangle.
\]
Applied to the three copies of $X$, this gives us three polynomials $p_1, p_2, p_3$ such that 
\[
p_1(\Phi^\tau_{d^1_0}) = 0,
\quad
p_2(\Phi^\tau_{d^2_0}) = 0,
\quad
\mbox{and}
\quad
p_1(\Phi^\tau_{d^3_0}) = 0.
\]
The images of the three roots in $\Ss(Y)$ obey the algebraic relation
\[
\Phi^\tau_{e_3} = \Phi^\tau_{e_1} + \Phi^\tau_{e_2}
\]
which is again an immediate consequence of $(**)$. This puts the double-resultant construction \ref{6s1}(5) on a geometric footing. (Our construction of the relation (R) obeyed by $\Phi^\tau_{d_0}$ over $\Aa(X)$ remains ad hoc.)

\end{sremark}

\section{The $p$-adic unipotent Albanese map}

Throughout this section, Tannakian fundamental groups are endowed with the usual \emph{functorial} product.

\subsection{}
\label{wacabuc1} \label{zcoords}
Let
\[
Y = \Spec \ZZ[z_1, z_2, z_1\inv, z_2\inv, (1-z_1)\inv, (1-z_2)\inv, (1-z_1 z_2)\inv],
\]
let $p$ be a prime, let $\pi^{\PL, \dR}(Y_\Qp)$ denote the polylogarithmic quotient of the de Rham unipotent fundamental group of $Y$ at the tangential base-point ``$(1,1)$ at $(0,0)$'' with respect to functorial composition of paths (\S\ref{polquosection}). Let 
\[
S = \{e_1, e_{11}, e_2, e_{22}, e_{12}\}.
\]
In \S\ref{cca3} we outlined the construction of the standard presentation 
\[
\pi(S)_{\Qp} 
\surj
\pi^{\PL, \dR}(Y_\Qp).
\]
We define
\[
\Lambda_{\ge -\infty}^\m{fun} = \bigcup_{i=1}^\infty \Lambda^\m{fun}_{-i},
\mbox{ where }
\Lambda_{-1}^\m{fun} = S,
\]
\[
\mbox{ and }
\Lambda_{-i}^\m{fun} = \{ e_1^{i-1}e_{11}, e_2^{i-1}e_{22}, (e_1+e_2)^{i-1}e_{12}\}
\mbox{ for } i \ge 2.
\]
If $\om$ is a finite linear combination of words in the alphabet $S$, we let $f_\om$ denote the linear functional on the completed universal enveloping algebra $\Uu(S)$ dual to $\om$ with respect to the standard (topological) basis. According to Lemma \ref{ab3} (applied to the opposite group), the functions $f_\la$ (for $\la \in \Lambda^\m{fun}_{\ge - \infty}$) on $\pi(S)_\Qp$ factor through $\pi^{\PL, \dR}(Y_\Qp)$ and form an algebra basis for its coordinate ring
\[
A^{\PL, \dR}(Y_\Qp) := \Oo(\pi^{\PL, \dR}(Y_\Qp)).
\]
Let $f_\la^\m{BC}$ denote the Besser-Coleman function on $Y(\Zp)$ obtained by composing $f_\la$ with the unipotent Albanese map
\[
\al: Y(\Zp) \to 
\pi^{\PL, \dR}(Y_\Qp). 
\]
Let $e_3 := e_1+e_2$.

\begin{sprop}
\label{albcoords}
In the situation and the notation of \S\ref{zcoords}, we have (for $i \ge 1$)
\begin{align*}
f_{e_1}^\m{BC}(z_1, z_2) &= \log(z_1) \\
f_{e_2}^\m{BC}(z_1, z_2) &= \log(z_2) \\
f_{e_1^{i-1}e_{11}}^\m{BC}(z_1, z_2) 
	&= \Li_i(z_1) \\
f_{e_2^{i-1}e_{22}}^\m{BC}(z_1, z_2) 
	&= \Li_i(z_2) \\
f_{e_3^{i-1}e_{12}}^\m{BC}(z_1, z_2) 
	&= \Li_i(z_1 z_2). \\
\end{align*}
\end{sprop}

\begin{proof}
Let $X = \Spec \ZZ[x,x\inv, (1-x)\inv]$ and consider the maps $p,q:Y \to X$ given by
\begin{align*}
(z_1, z_2) \overset{p}\mapsto z_1 &&
(z_1, z_2) \overset{q}\mapsto z_2.
\end{align*}
These maps induce maps of punctured tangent spaces at $(z_1,z_2) = (0,0)$ and send the tangential base-point ``$(1,1)$ at $(0,0)$'' to the tangential base-point ``$1$ at $0$''. The unipotent de Rham fundamental group $\pi_1^\dR(X_\QQ, \vec{1_0})$ is freely generated by two elements $d_0$ (monodromy about $x = 0$) and $d_1$ (monodromy about $x=1$). The maps induced by $p,q$ on $\pi_1$ send
\begin{align*}
e_1 \overset p \mapsto d_0 && e_1 \overset q \mapsto 0  \\
e_{11} \mapsto d_1 && e_{11} \mapsto 0  \\
e_2 \mapsto 0 && e_2 \mapsto d_0  \\
e_{22} \mapsto 0 && e_{22} \mapsto d_1  \\
e_{12} \mapsto 0 && e_{12} \mapsto  0.  \\
\end{align*}
We now focus on the map $p$ and the functions $f^\m{BC}_{e^{i-1}_1 e_{11}}$. By the formula given above for the induced map on fundamental groups, $f_{d_0^{i-1}d_1}$ pulls back along $p$ to $f_{e^{i-1}_1 e_{11}}$. On the other hand $f_{d_0^{i-1}d_1}$ pulls back along the $p$-adic unipotent Albanese to $\Li_i(x)$. The $p$-adic unipotent Albanese maps fit into a commuting square
\[
\xymatrix{
Y(\Zp) \ar[d]_p \ar[r] & \pi_1^\dR(Y_\Qp, (1,1)_{(0,0)}) \ar[d]
\\
X(\Zp) \ar[r] & \pi_1^\dR(X_\Qp, 1_0).
}
\]
Combining these facts we find that 
\[
f_{e_1^{i-1}e_{11}}^\m{BC}(z_1, z_2) 
= \Li_i(z_1),
\]
and similarly for $f_{e_2^{i-1}e_{22}}^\m{BC}$.

We turn to the function $f_{e_3^{i-1}e_{12}}^\m{BC}$. Let
\[
\Uu = \QQ \langle \langle S \rangle \rangle / I
\]
where $I$ is the two-sided ideal generated by the Lie relations \ref{1.1}(R). Let $\Ee = \Uu \otimes \Oo_Y$ with connection 
\[
\nabla: \Ee \to \Ee \otimes \Om^1_Y
\]
given on a word $W$ in the alphabet $S$ regarded as a section of the trivial pro-\textit{vector bundle} $\Ee$ by (notation as in \S\ref{cca3})
\[
\nabla(W) = -e_1 W \xi_1 - e_{11} W \xi_{11} -
e_{2} W \xi_{2} - e_{22} W \xi_{22} - e_{12} W \xi_{12}.
\]
Then $(\Ee, \nabla)$ is isomorphic to the universal unipotent connection on $Y$ (at any base-point), equipped with its de Rham trivialization. Hence, $f_{e_3^{i-1}e_{12}}^\m{BC}$ may be represented by the abstract Coleman function given by the connection $(\Ee, \nabla)$, the projection $f_{e_3^{i-1}e_{12}}: \Ee \to \Oo$, and the Frobenius-compatible family of horizontal sections on residue polydisks with \emph{constant term} $0$ at the tangential base-point $(1,1)_{(0,0)}$ \cite{BesserFurusho}; this is the same, \textit{mutatis mutandis}, as the case of $\thrpl$ treated, for instance, in Theorem 2.3 of Furusho \cite{furushoii}. It follows that the functions $f^{\m{BC}}_W$ obey 
\[
d \sum_W f^\m{BC}_W W
= \sum_V f^\m{BC}_V (e_1 W \xi_1 + e_{11} W \xi_{11} +
e_{2} W \xi_{2} + e_{22} W \xi_{22} + e_{12} W \xi_{12}).
\]
Hence $f^\m{BC}_{e_{12}}$ satisfies the differential equation 
\[
df^\m{BC}_{e_{12}} = \frac{d(z_1 z_2)}{1-z_1z_2}.
\]
Since $\Li_1(z_1 z_2)$ satisfies the same differential equation and has constant term $0$ at the base-point $(1,1)_{(0,0)}$, it follows that 
\[
f^\m{BC}_{e_{12}} = \Li_1(z_1 z_2).
\]
Similarly, for $i \ge 1$, $f_{e_3^{i}e_{12}}^\m{BC}$ satisfies the differential equation 
\[
d f_{e_3^{i}e_{12}}^\m{BC} 
= f_{e_3^{i-1}e_{12}}^\m{BC} \frac{d(z_1z_2)}{z_1 z_2},
\]
and (by induction), $\Li_{i+1}(z_1z_2)$ satisfies the same differential equation. Since $\Li_{i+1}(z_1z_2)$ too has constant term $0$, it follows that 
\[
f_{e_3^{i-1}e_{12}}^\m{BC}(z_1, z_2) 
	= \Li_i(z_1 z_2)
\]
as claimed. 
\end{proof}

\section{Summary and construction of $F^\m{BC}$} 
\label{wacasection}

\subsection{}\label{waca2}
Fix a prime $p$ not dividing $6$. In \S\ref{6s1} we constructed a polynomial $F$ in the 14 variables
\[
\tag{A$_\m{lex}$}
f_{e_1}, f_{e_{11}},\dots
\]
listed in \S\ref{gs1} whose coefficients are rational functions in the 11 symbols
\[
\tag{G$_\m{lex}$}
f_\tau, f_\upsilon, \dots
\]
(also listed in \S\ref{gs1}) over the rationals. Using the equations obtained in \S\ref{deceq} and replacing motivic polylogarithms by their $p$-adic periods, we obtain a polynomial with coefficients in $\Qp$. In terms of the coordinates $z_1, z_2$ on $Y = M_{0,5}$ (\S\ref{zcoords}), we replace the indeterminates by Besser-Coleman functions on $Y(\Zp)$ as follows:
\begin{align*}
f_{e_1} & \mapsto \log(z_1) \\
f_{e_2} & \mapsto \log(z_2) \\
f_{e_{11}e_1^{i-1}} & \mapsto \Li_i(z_1) \\
f_{e_{22}e_2^{i-1}} & \mapsto \Li_i(z_2) \\
f_{e_{12}e_3^{i-1}} & \mapsto \Li_i(z_1 z_2). \\
\end{align*}
This gives us a Besser-Coleman function $F^\m{BC}$ on $M_{0,5}(\Zp)$. Our goal for this section is to explain how the results obtained above show that $F^\m{BC}$ is within $\ep$ of a Kim function for $M_{0,5}$ in half-weight $4$ over $Z = \Spec \ZZ[1/6]$ (Theorem \ref{fbctheorem}) while clarifying $\ep$ and indicating how to apply our lexicographic computations to functorial fundamental groups.

\subsection{}\label{waca3}
Let $\pi_1^{\MT}(Z) = \pi_1^\un(Z) \rtimes \Gm$ denote the \emph{functorial}  fundamental group of the category of mixed Tate motives over $Z$ at the de Rham fiber functor and let $A(Z) = \Oo(\pi_1^\un(Z))$ be the associated graded Hopf algebra. Let $X = \thrpl$, let $\pi_1^\un(X, 1_0)$ denote the \emph{functorial} unipotent fundamental group of $X$ at the standard $\ZZ$-integral base point $1_0$. Let $\dR^*\pi_1^\un(X, 1_0)$ denote its de Rham realization. Let $d_0$, $d_1$ denote the standard generators on the latter. In view of the canonical trivializations of de Rham path torsors, a word $\om$ in $d_0, d_1$ gives rise to a function $f_\om$ on any path torsor. Recall that given $a \in X(Z)$ and $n \ge 1$ we define the \emph{(functorial, unipotent) motivic polylogarithm} $\Liu_n(a) \in A_n(Z)$ to be the function
\[
\pi_1^\un(Z) \xto{o(p^\dR)} \pi_1^\un(X, 1_0, a) 
\xto{f_{d_1d_0^{n-1}}}
\AA^1_\QQ.
\]
as in (\S\ref{lidef})$^\m{op}$.

\subsection{}\label{rome1}
Let $A(Z)_{[\le 4]} \subset A(Z)$ denote the subalgebra generated in half-weights $\le 4$ and by $\pi_1^\un(Z)_{\ge -4}$ the associated quotient of $\pi_1^\un(Z)$. We let $K(Z)_{[\le 4]}$ denote the fraction field of $A(Z)_{[\le 4]}$ and let
\[
\eta(Z)_{\ge -4} = \Spec K(Z)_{[\le 4]}.
\]
Let $K'(Z)_{[\le 4]} \subset K(Z)_{[\le 4]}$ denote the maximal localization of $A(Z)_{[\le 4]}$ to which the $p$-adic period map $\m{per}: A(Z) \to \Qp$ extends (conjecturally $K' = K$) and let
\[
\eta'(Z)_{\ge -4} = \Spec K'(Z)_{[\le 4]}.
\]
We denote the map of schemes
\[
\Spec \Qp \to \eta'(Z)_{\ge -4}
\]
induced by the period map by $I_{BC}$.

\subsection{}\label{waca4}
The nonabelian cohomology variety 
\[
\mathbf{H}^1(\pi_1^\un(Z, \om), \pi^\m{PL}_{\ge -4}(Y)_\om)
\]
is independent of the choice of fiber functor $\om$, which we therefore omit from the notation, and similarly for its filtered $\phi$ variant. Let $Z_p = \Spec \Zp$ and let
\[
\pi_1^\m{MT}(Z_p) = \pi_1^\m{un}(Z_p) \rtimes \Gm
\]
denote the fundamental group of the category of mixed Tate filtered $\phi$-modules over $\Qp$ \cite{mtmue} at the de Rham fiber functor. We denote the realization of a mixed Tate motive $M$ in mixed Tate filtered $\phi$-modules by $\opnm{F\phi^*}(M)$, and we denote de Rham realization by $\m{dR}^*(M)$. In the diagrams below, we let $\bf{RL}$ denote the map of nonabelian cohomology varieties obtained by realization and localization. We let $\ka, \ka_p$ denote the unipotent motivic and filtered $\phi$ Kummer maps. We replace $\bf{RL}$ by $RL$ and $\bf{H}$ by $H$ to denote the induced map of $\Qp$-points. This completes the definition of the objects and morphisms in the first diagram:
\[
\tag{*}
\xymatrix{
Y(Z) \ar[r] \ar[d]_\ka & Y(Z_p) \ar[d]^{\ka_p}
\\
H^1 \big(\pi_1^\m{MT}(Z)_\Qp, \pi^\m{PL}_{\ge -4}(Y)_\Qp \big)
\ar[r]_-{RL}
& H^1\big(\pi_1^\m{MT}(Z_p), F\phi^*\pi^\m{PL}_{\ge -4}(Y) \big).
}
\]
For a fuller discussion of a direct analog of this diagram: its commutativity and its (close) relationship to Kim \cite{kimi, kimii}, we refer the reader for instance to \cite{mtmue}.

\subsection{}\label{rome2}
By Proposition \ref{polquoprop}, the unipotent radical of $\pi_1^\m{MT}(Z)$ acts trivially on $\pi^\m{PL}_{\ge -4}(Y)$. Consequently, Proposition 5.2.1 of \cite{mtmue} applies to show that the natural map
\[
r: {\bf H}^1\big(\pi_1^\m{MT}(Z), \pi^\m{PL}_{\ge -4}(Y) \big)
\to
{\bf Z}^1\big(\pi_1^\m{un}(Z), \pi^\m{PL}_{\ge -4}(Y) \big)^\Gm
\]
to the $\QQ$-scheme (or functor) parametrizing $\Gm$-equivariant 1-cocycles is an isomorphism. We also have the usual isomorphism
\[
c: 
{\bf H}^1\big( \pi_1^\m{MT}(Z_p), F\phi^*\pi^\m{PL}_{\ge -4}(Y) \big)
\to
 \m{dR}^*\pi^\m{PL}_{\ge -4}(Y)_\Qp
\]
\cite{kimii, mtmue}, which follows from the fact that every $\pi_1^\m{MT}(Z_p)$-equivariant $F\phi^*\pi^\m{PL}_{\ge -4}(Y)$-torsor possesses a unique Frobenius-fixed point and a unique point in filtered degree 0. 

Since every $\Gm$-equivariant cocycle
\[
c: \pi_1^\un(Z)
\to 
\pi^\PL_{\ge -4}(Y)
\]
factors (uniquely) through $\pi_1^\un(Z)_{\ge -4}$, we have a canonical isomorphism
\[
{\bf Z}^1\big(\pi_1^\m{un}(Z), \pi^\m{PL}_{\ge -4}(Y) \big)^\Gm
=
{\bf Z}^1\big(\pi_1^\m{un}(Z)_{\ge -4}, \pi^\m{PL}_{\ge -4}(Y) \big)^\Gm
\]
and hence a \emph{universal cocycle evaluation map}
\[
\ev: 
\pi_1^\m{un}(Z)_{\ge -4}
\times
{\bf Z}^1\big(
\pi_1^\m{un}(Z),
\pi^\m{PL}_{\ge -4}(Y)
\big)^\Gm
\to
\pi_1^\m{un}(Z)_{\ge -4}
\times
\pi^\m{PL}_{\ge -4}(Y)
\]
given on points by
\[
\ev(\gamma, c) = (\gamma, c(\gamma)).
\]
We may then base-change $\ev$ along the evident maps
\[
\xymatrix{
&
\pi_1^\un(Z)_{\ge -4}
\\
\eta(Z)_{\ge -4} \ar[r] & 
\eta'(Z)_{\ge -4} \ar[u] &
\Spec \Qp \ar[l]
}
\]
(\S\ref{rome1}). We denote the base-change to $\eta'(Z)_{\ge -4}$ by $\ev'$, we denote the base-change to $\Spec \Qp$ by $\ev_{I_{BC}}$, and we denote the base-change to $\eta(Z)_{\ge -4}$ simply by $\ev$. This completes our definitions of the objects and morphisms in the following diagram, whose commutativity is clear.
\[
\tag{**}
\xymatrix{
{\bf H}^1\big(\pi_1^\m{MT}(Z), \pi^\m{PL}_{\ge -4}(Y) \big)_{\Qp}
\ar[r]_{\bf{RL}} \ar[d]_-{r}
& {\bf H}^1\big( \pi_1^\m{MT}(Z_p), F\phi^*\pi^\m{PL}_{\ge -4}(Y) \big)
\ar[d]^{c}
\\
\Spec \Qp \times 
{\bf Z}^1\big(\pi_1^\m{un}(Z), \pi^\m{PL}_{\ge -4}(Y) \big)^\Gm
\ar[r]_-{\fk {ev}_{I_{BC}}} \ar[d]_{I_{BC}}
& 
\Spec \Qp \times
 \m{dR}^*\pi^\m{PL}_{\ge -4}(Y)
\ar[d]^-{I_{BC}}
\\
\eta'(Z)_{\ge -4} \times 
{\bf Z}^1\big(\pi_1^\m{un}(Z), \pi^\m{PL}_{\ge -4}(Y) \big)^\Gm
\ar[r]_-{\fk{ev}'}
& 
\eta'(Z)_{\ge -4} \times
 \m{dR}^*\pi^\m{PL}_{\ge -4}(Y)
 \\
\eta(Z)_{\ge -4} \times 
{\bf Z}^1\big(\pi_1^\m{un}(Z), \pi^\m{PL}_{\ge -4}(Y) \big)^\Gm
\ar[r]_-{\fk{ev}}
\ar[u]
& 
\eta(Z)_{\ge -4} \times
 \m{dR}^*\pi^\m{PL}_{\ge -4}(Y).
\ar[u]
}
\]

\subsection{}
Since Lyndon words provide an algebra-basis for the shuffle algebra, the \emph{arithmetic shuffle coordinates} (A$_\m{fun}$) obtained by reversing the order of letters in \ref{waca2}(A$_\m{lex}$), form an algebra basis of $A(Z)_{[\le 4]}$. The morphism \ref{cca3}(*) provides a presentation of $\pi_1^\un(Y)$ (at any base-point) with Lie-algebra relations given by \ref{1.1}(R) \cite{OiUeno}. According to Lemma \ref{ab3}, the \emph{geometric shuffle coordintes} (G$_\m{fun}$) obtained by reversing the order of letters in \ref{waca2}(G$_\m{lex}$), form an algebra basis of
\[
A^\m{PL}_{[\le 4]}(Y) = \Oo ( \dR^* \pi^\m{PL}_{\ge -4}(Y)).
\]
In this way, $F$ defines a function on
\[
\eta(Z)_{\ge -4} \times
 \m{dR}^*\pi^\m{PL}_{\ge -4}(Y).
\]

Propositions (\ref{coocoo4})$^\m{op}$ and (\ref{6s2})$^\m{op}$ show that $F$ vanishes on the image of the evaluation map $\ev$. The computations of (\S\ref{sobsection})$^\m{op}$ as summarized in (\S\ref{deceq})$^\m{op}$ allow us to replace the coefficients of $F$ by polynomials in motivic polylogarithms which are unramified over $Z$, at the cost of a possible $p$-adic error of size determined by the precision of the $p$-adic periods on which these computations depend. Numerical evaluation of the $p$-adic periods of the coefficients then shows that $F$ factors through 
\[
\eta'(Z)_{\ge -4} \times
 \m{dR}^*\pi^\m{PL}_{\ge -4}(Y)
\]
and it follows that $F$ vanishes on the image of $\ev'$. Pullback by $I_{BC}$ corresponds to replacing the coefficients in $F$ by their $p$-adic periods. Further, according to Proposition \ref{albcoords}, pullback by the unipotent Albanese map
\[
\al = c \circ \ka_p
\]
corresponds to the replacement of indeterminates by Besser-Coleman functions as listed at the end of \S\ref{waca2}. By the commutativity of \ref{rome2}(**), this shows that $F^\m{BC}$ (after possibly enduring a small modification) is a Kim function as claimed. This completes the proof of Theorem \ref{fbctheorem}.

\subsection{}
\label{addendum}
Some Kim functions on $M_{0,5}$ are uninteresting because they \emph{come from Kim functions on $M_{0,4}$}. Instead of making this notion precise (in one of several possible ways), we give a concrete example. As above, we identify $M_{0,4}$ with $X=\Spec \ZZ[x, x\inv, (1-x)\inv]$ and we identify $M_{0,5}$ with
\[
Y=\Spec \ZZ[z_1, z_2, z_1\inv, z_2\inv, (1-z_1)\inv, (1-z_2)\inv, (1-z_1z_2)\inv].
\]
The modular interpretations of these functions are determined by the formulas
\[
x(\PP^1, 0, 1, \infty, a)
= a,
\]
\[
\quad
z_1(\PP^1, 0, 1, \infty, c,d)
= \frac{c}{d},
\quad \mbox{and} \quad
z_2(\PP^1, 0, 1, \infty, c,d)
= d.
\]
In terms of our presentations, the map
\[
f: (\PP^1, 0, 1, \infty, c,d)
\mapsto
(\PP^1, 0, 1, \infty,d)
\]
corresponds to the second projection
\[
Y \to X
\]
which extends to 
\[
\Gm \times \Gm \to \Gm
\]
and respects our chosen tangential base-points. This means that it induces a $\Gm$-equivariant map of polylogarithmic quotients
\[
\pi^\PL(Y) \to \pi^\PL(X).
\]
Thus, for any $Z \subset \Spec \ZZ$, $n \in \NN$ and $p \in Z$, if $F$ is a Kim function on $X(\Zp)$ associated to $Z$ and to $\pi^\PL_{\ge -n}(X)$, then $f^\sharp F$ is a Kim function on $Y(\Zp)$ associated to $Z$ and to $\pi^\PL_{\ge -n}(Y)$.

But the function of Theorem \ref{fbctheorem} is not of this form. Indeed, a simple dimension count shows that (regardless of $p$) there are no Kim functions for $X$ over $Z = \Spec \ZZ[1/6]$ in half-weight $n=4$.

\bibliography{M05_Refs}
\bibliographystyle{alphanum}

\vfill

\Small

\noindent
\textsc{Ishai Dan-Cohen} 

\noindent 
\textsc{Department of mathematics}

\noindent
\textsc{Ben-Gurion University of the Negev}

\noindent
\textsc{Be'er Sheva, Israel}

 \noindent
\textsc{Email address:} \texttt{ishaidc@gmail.com}

\bigskip

\noindent
\textsc{David Jarossay} 

\noindent 
\textsc{Department of mathematics}

\noindent
\textsc{Ben-Gurion University of the Negev}

\noindent
\textsc{Be'er Sheva, Israel}

 \noindent
\textsc{Email address:} \texttt{jarossay@post.bgu.ac.il}

\end{document}